\documentclass{amsart}

\usepackage[mathscr]{eucal}
\usepackage{amsmath, amsthm, amssymb, amsfonts, mathrsfs, amscd, enumerate, xcolor, layout, stmaryrd, float, longtable, mathtools} 
\usepackage[pdftex]{graphicx}
\usepackage{subcaption} 
\usepackage{verbatim, enumitem, nicefrac}

\newtheorem{theorem}{Theorem}[section]
\newtheorem{corollary}[theorem]{Corollary}
\newtheorem{lemma}[theorem]{Lemma}
\newtheorem{proposition}[theorem]{Proposition}

\theoremstyle{definition}
\newtheorem{definition}[theorem]{Definition}
\newtheorem{remark}[theorem]{Remark}

\newtheorem{example}[theorem]{Example}

\newcommand{\vungoc}{V\~u Ng\d{o}c}

\def\epsilon{\varepsilon}
\def\phi{\varphi}

\newcommand{\al}{\alpha}
\newcommand{\be}{\beta}

\newcommand{\ze}{\zeta}

\newcommand{\lam}{\lambda}

\newcommand{\om}{\omega}

\def\D{{\mathbb D}}

\def\R{{\mathbb R}}
\def\mbS{{\mathbb S}} 

\def\Z{{\mathbb Z}}

\newcommand{\x}{\times}
\newcommand{\del}{\partial}

\newcommand{\lapla}{\triangle}

\DeclareMathOperator{\rk}{rk}

\def\slashii#1{\setbox0=\hbox{$#1$}             
\dimen0=\wd0                                 
\setbox1=\hbox{\sl/} \dimen1=\wd1            
\ifdim\dimen0>\dimen1                        
\rlap{\hbox to \dimen0{\hfil\sl/\hfil}}   
#1                                        
\else                                        
\rlap{\hbox to \dimen1{\hfil$#1$\hfil}}   
\hbox{\sl/}                               
\fi}                                         %
\def\slashiii#1{\setbox0=\hbox{$#1$}#1\hskip-\wd0\hbox to\wd0{\hss\sl/\/\hss}}
\newcommand{\refthmIntro}{Theorem \ref{thmIntro}}

\newcommand{\refmorse}{Corollary \ref{morse}}

\newcommand{\refcoupledAngular}{Example \ref{coupledAngular}}

\newcommand{\refcoupledSpin}{Example \ref{coupledSpin}}

\newcommand{\refnondegBF}{Lemma \ref{nondegBF}}

\newcommand{\refthmEliasson}{Theorem \ref{thmEliasson}}

\newcommand{\refmainThm}{Theorem \ref{mainThm}}

\newcommand{\refnoPole}{Lemma~\ref{noPole}}  

\newcommand{\refcritQuotient}{Lemma \ref{critQuotient}}

\setenumerate[1]{label={\arabic*)}}

 \usepackage[colorlinks=true]{hyperref}
\hypersetup{urlcolor=blue, citecolor=red}

\title[Semitoric systems with two focus-focus singularities]{A family of compact semitoric systems with two focus-focus singularities}

\author[Sonja Hohloch and Joseph Palmer]{}


 \email{sonja.hohloch@uantwerpen.be}
 \email{j.palmer@math.rutgers.edu}

\thanks{The first author was partially supported by the Research Fund of the University of Antwerp and the second
author is partially supported by an AMS-Simons travel grant.}

\begin{document}
\maketitle
\centerline{\scshape Sonja Hohloch}
\medskip
{\footnotesize
 \centerline{University of Antwerp}
   \centerline{Department of Mathematics and Computer Science}
   \centerline{Middelheimlaan 1}
   \centerline{ B-2020 Antwerpen, Belgium}
} 

\medskip

\centerline{\scshape Joseph Palmer}
\medskip
{\footnotesize
 \centerline{Rutgers University}
 \centerline{Department of Mathematics}
 \centerline{Hill Center - Busch Campus}
 \centerline{110 Frelinghuysen Road}
 \centerline{Piscataway, NJ 08854, USA.}
}

\bigskip


\begin{abstract}
About 6 years ago, semitoric systems were classified by Pelayo $\&$ \vungoc\ by means of five invariants. Standard examples are the coupled spin oscillator on $\mbS^2 \x \R^2$ and coupled angular momenta on $\mbS^2 \x \mbS^2$, both having exactly one focus-focus singularity.
But so far there were no {\em explicit} examples of systems with more than one focus-focus singularity
which are semitoric in the sense of that classification.
This paper introduces a $6$-parameter family of integrable systems on $\mbS^2 \x \mbS^2$ and proves that, for certain ranges of the parameters, it is a compact semitoric system with precisely two focus-focus singularities.
Since the {\em twisting index} (one of the semitoric invariants) 
is related to the relationship between different focus-focus points,
this paper provides systems for the future study of the twisting index.
\end{abstract}


\section{Introduction}

An \emph{integrable system} is a triple $(M,\om,F)$ where $(M,\om)$ is a $2n$-dimensional symplectic manifold and
$F\colon M\to \R^n$ is a smooth function, known as the \emph{momentum map}, whose components Poisson commute and
are linearly independent almost everywhere.  The points at which the linear independence fails are known as 
\emph{singular points}.
An integrable system is \emph{toric} if $M$ is compact and the Hamiltonian vector fields of the components
all have periodic flow of the same period; 
in this case the image of the momentum map $F(M)$
is a convex $n$-dimensional polytope (a special case of the Atiyah-Guillemin-Sternberg Theorem~\cite{Atiyah1982, GuilStern1982})
and additionally, by the work of Delzant~\cite{Delzant1988}, $F(M)$ completely determines the system $(M,\om,F)$ up
to equivariant symplectomorphism.

So-called semitoric integrable systems are a special class of integrable systems on $4$-manifolds for which
one of the two components of its momentum map has a Hamiltonian vector field with periodic flow.
Specifically, a \emph{semitoric integrable system} is an integrable system $(M,\om, F=(J,H))$
such that $J$ is proper with periodic flow and every singular point is nondegenerate with no
hyperbolic blocks (see Section~\ref{sec_intsystem_singpts} for a discussion of types of singular points).  
Semitoric integrable systems can have singular points of focus-focus type, which
do not occur in toric integrable systems, and are an example
of almost toric manifolds which were introduced by Symington~\cite{Sy2003}.

Semitoric integrable systems were studied and classified
by Pelayo $\&$ \vungoc\ ~\cite{PVNinventiones,PVNacta}. 
The classification is in terms of five invariants: the number of 
focus-focus points (which is finite according to \vungoc~\cite{VuNgoc07}); an 
infinite family of polygons known
as a semitoric polygon; a Taylor series in two variables for each focus-focus point; the 
height of the focus-focus value in the semitoric polygon;
and the twisting index, which, roughly, is an integer for each pair of focus-focus points describing the `twist' of the singular Lagrangian fibration
between them. Semitoric systems are rigid enough to admit a classification, but flexible enough to appear more frequently in physical examples and to admit more interesting dynamics.  The main reason semitoric systems exhibit more interesting behavior than toric systems is the presence of the focus-focus points and the monodromy that these
singularities can produce in the integral affine structure of the momentum map image $F(M)$.

While the Pelayo-\vungoc\ classification predicts many systems and gives certain properties of those
systems, one thing that has thus far been lacking are {\em explicit} examples
of semitoric systems giving the symplectic manifold $(M,\om)$ and the momentum map $F$.
Le Floch $\&$ Pelayo~\cite{LFPecoupledangular} explicitly describe the coupled angular momenta system (originally 
described in~\cite{SaZh-PhysLettersA}, see \refcoupledAngular)
and details of the coupled spin oscillator (see \refcoupledSpin) are spread over several papers.
These systems each have exactly one focus-focus singularity.
In the present work we describe semitoric systems on $M=\mbS^2 \x \mbS^2$ which have two focus-focus singular points,
generalizing the system from~\cite{LFPecoupledangular}.
More precisely, the main result of this paper is the following.

\begin{theorem}
\label{thmIntro}
 Let $M=\mbS^2\x\mbS^2$ be equipped with the symplectic form $\om = -( R_1\om_{\mbS^2}\oplus R_2\om_{\mbS^2})$
 where $\om_{\mbS^2}$ is the standard volume form on the sphere
 and $0<R_1<R_2$ are real numbers.
 For $\vec{R}:=(R_1, R_2)$ and $\vec{t}:=(t_1, t_2, t_3, t_4)\in\R^4$ define $J_{\vec{R}},H_{\vec{t}} \ \colon M\to\R$ by
 \begin{equation}\label{eqn_thesystem}
  \begin{cases}
   J_{\vec{R}} (x_1, y_1, z_1, x_2, y_2, z_2): = R_1 z_1 + R_2 z_2,\\
   H_{\vec{t}} \ (x_1, y_1, z_1, x_2, y_2, z_2): = t_1 z_1 + t_2 z_2 + t_3 (x_1 x_2 + y_1 y_2) + t_4 z_1 z_2
  \end{cases}
 \end{equation}
 where $(x_i, y_i, z_i)$ are Cartesian coordinates on $\mbS^2 \subset \R^3$ for
 $i=1,2$.
 Then there exist choices of $t_1, t_2, t_3, t_4, R_1, R_2$ such that
 $(M,\om,(J_{\vec{R}},H_{\vec{t}}))$ is a semitoric system with exactly
 two focus-focus points.
\end{theorem}

\noindent
\refthmIntro\ is restated in more detail in Section~\ref{sec_familyofsystems} as \refmainThm.
The coupled angular momenta system with coupling parameter $t\in \ ]0,1[$ is the special case of Equation~\eqref{eqn_thesystem} 
with $t_1 = t$, $t_3 = t_4 = 1-t$, and $t_2=0$. 
The coupled angular momenta system describes the rotation of two vectors
(with magnitudes $R_1$ and $R_2$) about the $z$\--axis and has as a second integral
a linear combination of the $z$\--component of the first vector and the inner produce of 
the two vectors, while the system in Equation~\eqref{eqn_thesystem} includes additionally the $z$\--component
of the second vector and also breaks the inner product into two components, namely the projection to the $z$\--axis and the projection to the $xy$\--plane.

The system in Equation~\eqref{eqn_thesystem} is studied from a different point of view in mathematical physics,
where it is a special case of a generalized Gaudin model. We refer the interested reader to Petrera's PhD thesis \cite{petrera} and the references therein for the development since Gaudin's original work \cite{gaudin}.

 \refthmIntro\ gives explicit global formulas (defined by
the same expression on the entire manifold) for a family of examples of
semitoric systems with more than one focus-focus point. This family should
be useful for understanding semitoric systems at a concrete, computationally
amenable, context.
%
The twisting index invariant is
related to the relationship between different focus-focus singular points, so having an example
with multiple focus-focus points will help in understanding this invariant (though it does actually appear
in a more subtle way for systems with only one focus-focus point).

Additionally, not only the system itself, but also the method by which we produce this system is of interest.
We construct it as a linear combination of four different systems of toric type (semitoric systems with
no focus-focus points) and in this way one can see how it deforms into each of these four systems (see Figure~\ref{fig_array})
which correspond to four elements of the associated semitoric polygon.  
Let $N$ denote the north pole of $\mbS^2$ and $S$ denote the south pole, so that
$(N,N)$, $(N,S)$, $(S,N)$, and $(S,S)$ are the four possible products of poles in $\mbS^2\times\mbS^2$.
The next theorem
follows from Theorem~\ref{thm_s1s2full} in Section~\ref{sec_s1s2},
in which we take $R_1 = 1$ and $R_2 = 2$ for simplicity.
\begin{theorem}\label{thm_s1s2intro}
For $s_1,s_2\in [0,1]$ let $(J_{\vec{R}},H_{(s_1,s_2)})$ denote the system $(J_{\vec{R}},H_{\vec{t}})$
where
\[
t_1 = (1-s_1)(1-s_2),\qquad t_2 = s_1 s_2,\qquad t_3= s_1+s_2-2 s_1 s_2,\qquad t_4 = s_1-s_2.
\]
Then $(J_{(1,2)},H_{(s_1,s_2)})$ has the following properties: 
\begin{enumerate}
 \item it is an integrable system for all $(s_1,s_2)\in [0,1]^2$\emph{;}
 \item it is a semitoric system when $(s_1,s_2)\in[0,1]^2\setminus\gamma$ where $\gamma\subset [0,1]^2$ is the union of four smooth curves\emph{;}
 \item the points $(N,S),(S,N)\in\mathbb{S}^2\times\mathbb{S}^2$ transition between being elliptic-elliptic, focus-focus, and
  degenerate depending on the value of $(s_1,s_2)$\emph{;}
 \item it is semitoric with exactly two focus-focus points for all $(s_1,s_2)$ in an open neighborhood of $\left(\frac{1}{2},\frac{1}{2}\right)$\emph{;}
 \item it is semitoric with no focus-focus point if $(s_1,s_2)\in \{(0,0),(0,1),(1,0),(1,1)\}$.
\end{enumerate}
\end{theorem}
The set $\gamma$ represents the moment at which singular points become degenerate while they change
between focus-focus and elliptic-elliptic type.
Proposition~\ref{prop_degenerate} states that if the type
of a singular point changes from focus-focus to elliptic-elliptic
by smoothly varying the integrals (on a fixed manifold) then it must become degenerate 
during the transition, in fact, it is undergoing a Hamiltonian-Hopf bifurcation, see Remark~\ref{rmk_HH}.
The set $\gamma$ is an intersection of zero sets of discriminants of certain polynomials, see Equation~\eqref{def_gamma}.
The image of the momentum map for the system in Theorem~\ref{thm_s1s2intro} is plotted in
Figure~\ref{fig_array} for various choices of $(s_1,s_2)$ and $\gamma$ is plotted in Figure~\ref{fig_FFs1s2}. 
The coupled angular momenta system from Le Floch $\&$ Pelayo~\cite{LFPecoupledangular} is exactly the one parameter family of systems
obtained from the system in Theorem~\ref{thm_s1s2intro} by taking $s_2=0$, so the momentum map image of the coupled
angular momentum system is the bottom row of images in Figure~\ref{fig_array}.

\begin{figure}[h]
\includegraphics[width=350pt]{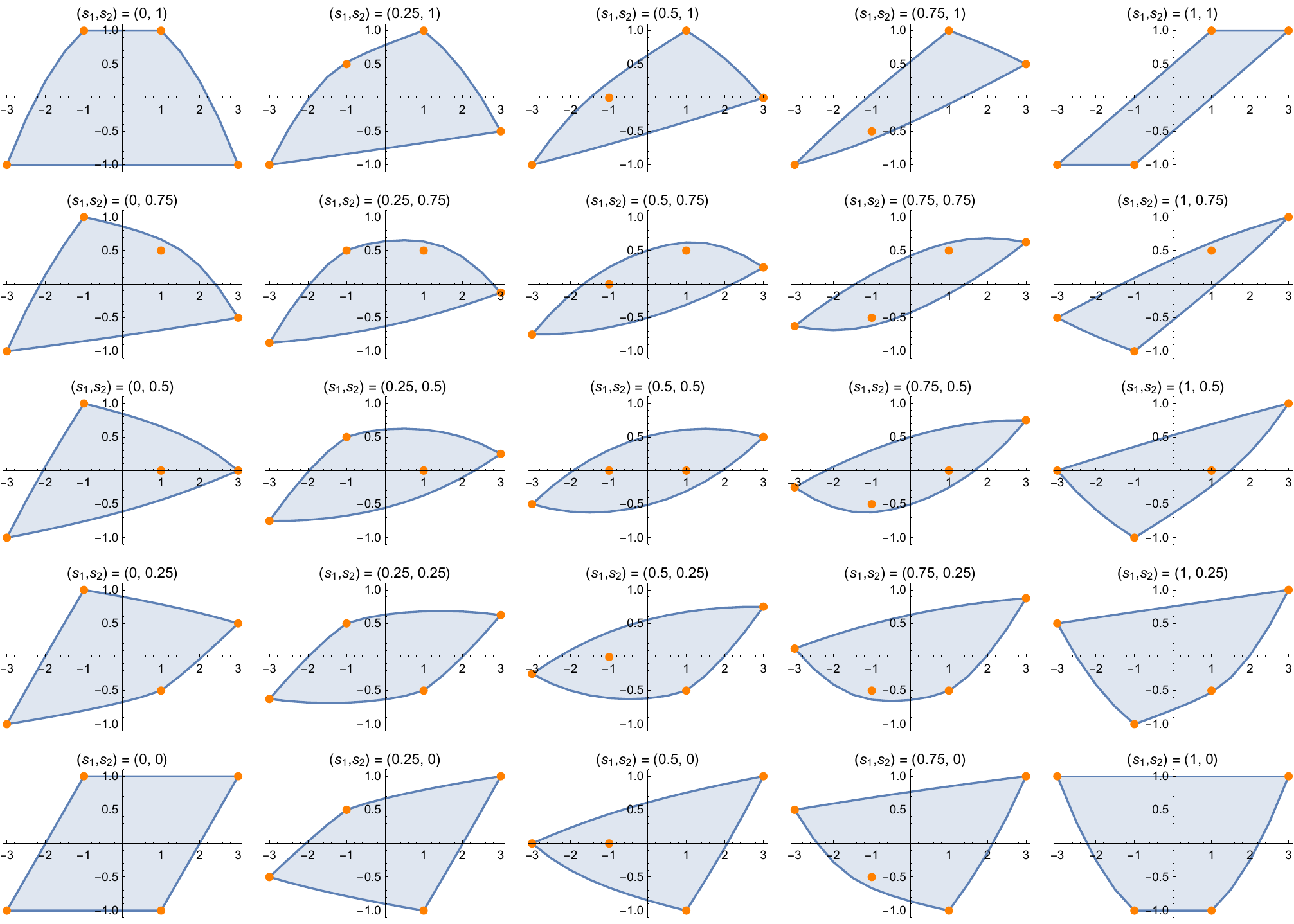}
\caption{An image of the momentum map $(J_{(1,2)}, H_{(s_1,s_2)})$ with the rank 0 points marked for varying values of $s_1, s_2\in [0,1]$.
Notice that the coupled angular momenta system shown in Figure~\ref{fig_coupledangular}
is the bottom row of the system shown in this figure since the coupled angular momenta is the special
case for which $s_2=0$.}
\label{fig_array}
\end{figure}

\begin{figure}[h]
\centering
    \begin{subfigure}[b]{0.4\textwidth}
        \includegraphics[width=\textwidth]{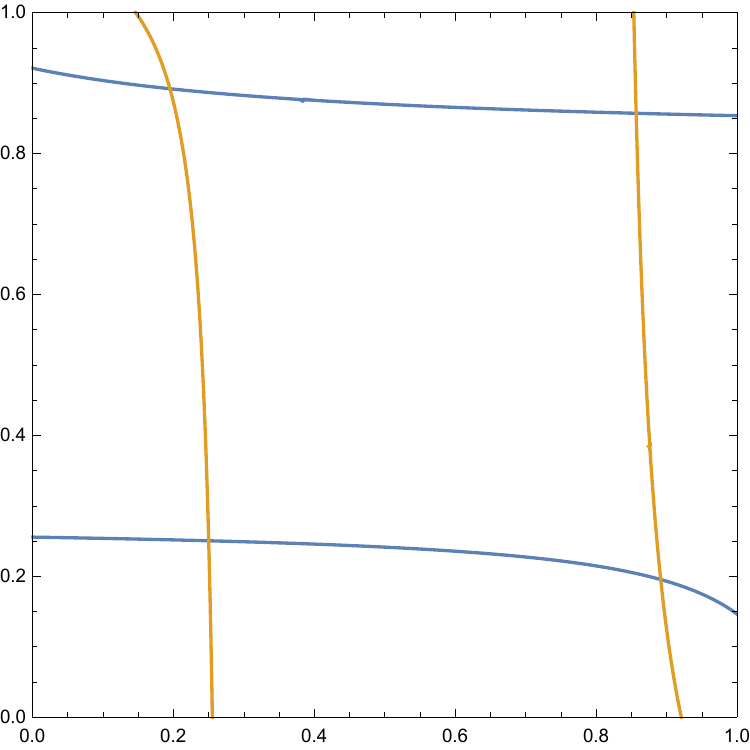}
    \end{subfigure}
    \qquad \quad
    \begin{subfigure}[b]{0.4\textwidth}
        \includegraphics[width=\textwidth]{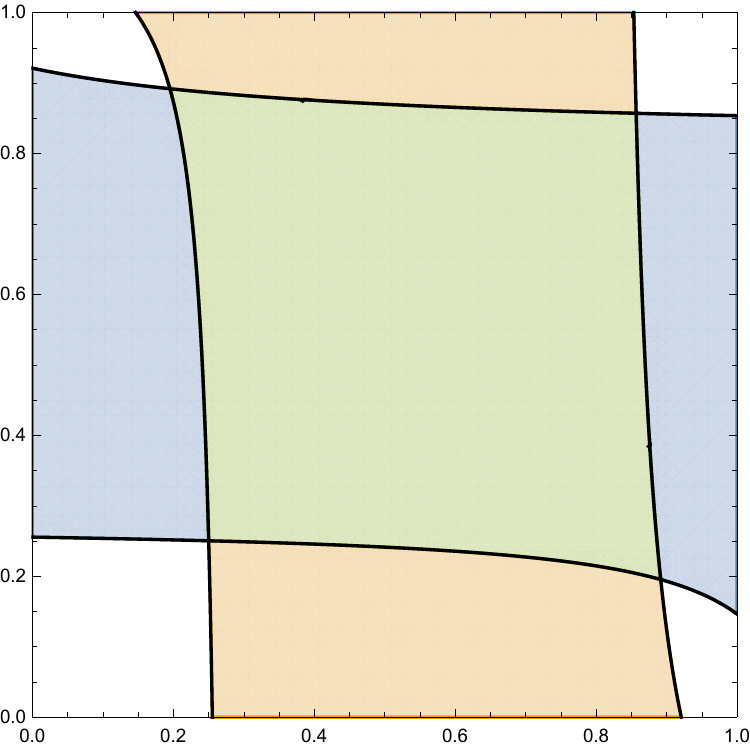}
    \end{subfigure}
 \caption{Left: a plot of the set $\gamma$, which is the union of $\gamma_{(S,N)}$ (blue)
 and $\gamma_{(N,S)}$ (orange), see Equation~\ref{def_gamma}. Right: Values of $(s_1, s_2)$ for which the system
$(J_{(1,2)}, H_{(s_1, s_2)})$ has focus-focus values at: only the point $(S,N)$ (blue),
only the point $(N,S)$ (orange), or at both points (green).  The system is degenerate 
on the black curves.
Compare with Figure~\ref{fig_array}.}
  \label{fig_FFs1s2}   
\end{figure}

Recently there has been a lot of activity relating to semitoric integrable systems, which we review briefly now.
There has been work regarding quantizations of semitoric integrable systems, specifically related to the
problem of recovering the classical system from the
quantum one (see for instance Le Floch $\&$ Pelayo $\&$ \vungoc~\cite{LFPeVN2016}). Hohloch $\&$ Sabatini $\&$ Sepe~\cite{HSS} answer the question of how the 
classification of semitoric systems is linked to Karshon's classification~\cite{karshon} of 
Hamiltonian $\mbS^1$-spaces. 
The question of lifting a Hamiltonian $\mbS^1$-action to a semitoric system is an ongoing project by Hohloch $\&$ Sabatini $\&$ Sepe $\&$ Symington and has been the topic of several conference talks.
There has been work to determine the convexity
of the momentum map image with respect to its intrinsic integral affine structure by Ratiu $\&$ Wacheux $\&$ Zung~\cite{RaWaZu2017}. 
Alonso $\&$ Dullin $\&$ Hohloch~\cite{AFDH-spin} are computing higher 
order terms of the Taylor series invariant of the focus-focus point in the coupled 
spin-oscillator (\refcoupledSpin\ of the present paper). 
Deformations of semitoric systems have been studied by endowing the moduli space
with a topology, see Palmer~\cite{PalmerJGP2017}. Kane $\&$ Palmer $\&$ Pelayo~\cite{KaPaPeSL2Z,KaPaPeminimal} used combinatorial methods to study blowups/downs and minimal models of semitoric systems. 
Generalizations of semitoric systems are achieved in Pelayo $\&$ Ratiu $\&$ \vungoc~\cite{PeRaVN2015} and 
Hohloch $\&$ Sabatini $\&$ Sepe $\&$ Symington~\cite{HSS-vertical}.
Additionally, work has begun to extend the theory of semitoric systems to
higher dimensional manifolds in Wacheux~\cite{Wacheux-thesis}.
Surgery techniques for semitoric systems are an ongoing project by Hohloch $\&$ Sabatini $\&$ Sepe $\&$ Symington.
Presently, hyperbolic singularities
are excluded from semitoric integrable systems, but Dullin $\&$ Pelayo~\cite{dullin-pelayo} have produced a smooth family of systems with transition from being semitoric to having a family of hyperbolic singular points.
A reader new to integrable systems can consult the surveys Pelayo $\&$ \vungoc~\cite{PVNsurvey} and Pelayo~\cite{Pesurvey},
or the books Marsden $\&$ Ratiu~\cite{MarsdenRatiu-book} and Cushman $\&$ Bates~\cite{CushmanBates-book}.

\subsubsection*{Structure of the article:} In Section~\ref{sec_intsystem_singpts} we review the required background, including
integrable systems, singular points, and semitoric integrable systems.  In Section~\ref{sec_familyofsystems}
we introduce the new system and prove \refthmIntro. In Section~\ref{sec_s1s2} we discuss the choice of parameters
for which the system can be seen as a linear combination of four systems of toric type, and
prove Theorem~\ref{thm_s1s2intro}.

\subsubsection*{Figures:} All figures and associated numerical computations in this article were made with the computer program {\em Mathematica}.


\section{Fundamental definitions}

\label{sec_intsystem_singpts}

In Section~\ref{sec_integrablesystemsintro} we introduce standard notions related to integrable systems
and non-degenerate points. A reader familiar with these topics can skip
directly to Section~\ref{sec_semitoricsystems}.

\subsection{Integrable systems and non-degenerate singular points}
\label{sec_integrablesystemsintro}

\subsubsection{Integrable systems}
Given a symplectic manifold $(M,\om)$ recall that associated to any function
$f\in C^\infty (M)$ there is a vector field denoted by $X^f$, called
the \emph{Hamiltonian vector field associated to $f$}, and defined by
\[
\om(X^f,\cdot) =- df(\cdot).
\]
Moreover, recall the Poisson bracket $\{\cdot, \cdot\}\colon C^\infty (M)\times C^\infty (M)\to C^\infty (M)$ given 
by $\{f,g\} = \om (X^f, X^g)$.
An \emph{integrable system} is a symplectic $2n$\--manifold $(M,\om)$ along with a collection
of functions $f_1, \ldots, f_n$ which Poisson commute (i.e.~$\{f_i,f_j\}=0$ for all $i,j$) and for which the associated Hamiltonian
vector fields $X^{f_1}, \ldots, X^{f_n}$ are linearly independent almost everywhere.
The function $F = (f_1, \ldots, f_n)\colon M\to\R^n$ is known as the
\emph{momentum map} of this system.

In this article, we will focus on the case $n=2$, so an integrable system
will be a $4$-dimensional symplectic manifold $(M,\om)$ with a function
$F\colon (M,\om)\to\R^2$ whose components $F=(J,H)$ are such that $\{J,H\}=0$ and
$X^J(p)$ and $X^H(p)$ are linearly independent for almost all $p\in M$.

The points at which linear independence of the components of the momentum
map fails are known as the \emph{singular points} of the system and the
\emph{rank} of a singular point is the rank of the differential
of the momentum map $dF$ at that point.  
There is a natural notion of non-degeneracy for such singular
points which we review now.
Rank 0
singular points are known as \emph{fixed points} since they are fixed under the
flow of the Hamiltonian vector fields of the components of the momentum map;
we will start with the classification of those.

\subsubsection{Rank 0 singular points, i.e., fixed points} Let $p\in M$ be a fixed point and
let $\mathcal{Q}(T_pM)$ denote the vector space of quadratic forms
on $T_pM$.
The symplectic form
on $M$ gives $\mathcal{Q}(T_p M)$
the structure of a Lie algebra
which is isomorphic to $\mathfrak{sp}(4,\R)$,
see Bolsinov $\&$ Fomenko~\cite{bolsinov-fomenko}.
Recall that a Cartan subalgebra is a nilpotent and self\--normalizing subalgebra.

\begin{definition}
\label{def_nondegenerate}
A fixed point $p\in M$ is \emph{non-degenerate} if the Hessians
$d^2J(p)$ and $d^2H(p)$ span a Cartan subalgebra
of the Lie algebra of quadratic forms on $T_pM$.
\end{definition}

In practice, this condition can be checked by use of the following lemma.

\begin{lemma}[Bolsinov $\&$ Fomenko~\cite{bolsinov-fomenko}]
\label{nondegCriterium}
\label{nondegBF}
 Let $p\in M$ be a fixed point. 
 Denote by $\om_p$ the matrix of the symplectic form with respect to a
 basis of $T_p M$ and let $d^2J$ and $d^2H$ denote the matrices of the Hessians
 of $J$ and $H$ with respect to the same basis.
 Then $p$ is non-degenerate if and only if $d^2J$ and
 $d^2H$ are linearly independent and there exists a linear
 combination of $\om_p^{-1}d^2 J$ and $\om_p^{-1}d^2 H$ which 
 has four distinct eigenvalues. 
\end{lemma}

\begin{proof}[Sketch of proof]
The result follows from the fact that an abelian subalgebra of $\mathfrak{sp}(4,\R)$ is a Cartan subalgebra if and only
if it is two dimensional and contains a regular element, in which case it is the centralizer for this regular element.
The span of $\om_p^{-1}d^2 J$ and $\om_p^{-1}d^2 H$
is an abelian subspace of $\mathfrak{sp}(T_p M) \cong \mathfrak{sp}(4,\R) \cong \mathcal{Q}(T_pM)$ because 
$J$ and $H$ Poisson commute (since they form an integrable system)
and a regular element is any matrix with four distinct eigenvalues.
We conclude that if $\om_p^{-1}d^2 J$ and $\om_p^{-1}d^2 H$ are linearly independent and
their span includes an element with four eigenvalues then the span
is a two\--dimensional abelian subalgebra which contains a regular element, and
is thus Cartan.
\end{proof}

\subsubsection{Rank 1 singular points} 
To define rank 1 non-degenerate singular points we will again follow Bolsinov $\&$ Fomenko~\cite[Section 1.8.2]{bolsinov-fomenko}.
Suppose that $p$ is a singular point of rank 1 in a 4\--dimensional integrable system $(M,\om,F=(J,H))$.
Then there exists some $\mu,\lam\in\mathbb{R}$ such that $\mu d H + \lam dJ = 0$ 
at $p$ and the $\mathbb{R}^2$\--action
defined by flowing along the vector fields of $J$ and $H$ has a one-dimensional orbit through $p$.
Let $L\subset T_pM$ be the tangent line of this orbit at $p$ and let $L'$ be the symplectic orthogonal complement to $L$.
Notice that $L\subset L'$ and since $J$ and $H$ Poisson commute they are invariant under the $\R^2$\--action
and thus the operator $\mu d^2H + \lam d^2J$
descends to the quotient $L'/L$.

\begin{definition}[Bolsinov $\&$ Fomenko~\cite{bolsinov-fomenko}]
 The rank 1 critical point $p$ is \emph{non-degenerate} if $\mu d^2H + \lam d^2J$
 is invertible on $L'/L$.
\end{definition}

\noindent
Now suppose that the flow of $X^J$ is periodic.  Recall that the symplectic quotient of $M$ by $J$ at the level $c$,
which we denote $M\sslash \mbS^1$, is the symplectic manifold $J^{-1}(c)/{\mbS^1}$	
where the ${\mbS^1}$\--action on $J^{-1}(c)$ is the one which comes from the flow of the Hamiltonian vector
field of $J$.

\begin{lemma}
\label{lem_defnondeg_rank1}
If $p\in M$ is a rank 1 singular point such that $dJ\neq 0$ then $p$ is non-degenerate if and 
 only if $d^2H$ is invertible at the image of $p$ in the symplectic quotient of $M$ by $J$ at the level $J(p)$.
\end{lemma}

\begin{proof}
 Let $L$ and $L'$ be as above and let $c=J(p)$.  Notice that $dJ\neq 0$ 
 and $\mathrm{dim}(L)=1$ implies that $L$ is spanned by $X^J$.
 Thus $v\in L'$ if and only if $\om_p(v, X^J) = 0$.  By the definition of the Hamiltonian vector field 
 this is equivalent to $v(J)=0$, so $v\in T_p(J^{-1}(c))$. Thus $L' = T_p(J^{-1}(c))$. Furthermore, $L$ is the tangent
 space to the orbit of the ${\mbS^1}$\--action through $p$ so $L'/L = T_p(J^{-1}(c)/{\mbS^1})$ and the result follows.
\end{proof}

\noindent
Lemma~\ref{lem_defnondeg_rank1} implies the following. 

\begin{corollary}
\label{morse}
If $dJ\neq0$ at all points of nonzero rank then
all rank 1 points of $(M,\om,F=(J,H))$ are non-degenerate if and only if 
$H$ descends to a Morse function on all possible symplectic quotients by $J$.
\end{corollary}

\noindent
See Bolsinov $\&$ Fomenko~\cite{bolsinov-fomenko} for a description of non-degenerate points for dimensions greater than four
and a description of rank 1 non-degenerate points in terms of Cartan subalgebras.

\subsubsection{Classification of non-degenerate points}
Williamson~\cite{Williamson1936} classified Cartan subalgebras of $\mathfrak{sp}(n,\R)$, 
which in turn implies a classification of the possible subalgebras $\mathfrak{c}$ generated
by the Hessians in $T_pM \cong \mathfrak{sp}(n,\R)$ at a non-degenerate singular points.
Eliasson~\cite{Eliasson-thesis} and Miranda $\&$ Zung~\cite{miranda-zung} extended Williamson's pointwise classification to a local
classification, which classifies the possible forms of the momentum map 
in local symplectic coordinates around a fixed point $p$, often known as the \emph{Eliasson-Miranda-Zung normal form}.

\begin{theorem}[Eliasson~\cite{Eliasson-thesis}, Miranda $\&$ Zung~\cite{miranda-zung}]
\label{thmEliasson}
 If $p\in M$ is a non-degenerate singular point of an $n$\--dimensional
 integrable system $(M,\om,F)$ then there exist local symplectic coordinates
 $(x,y):=(x_1, \ldots, x_n, y_1, \ldots, y_n)$ around $p$ such that there exist
 $q_1, \ldots, q_n\colon M\to\R$ where each $q_i$ is given by one of:
 \begin{enumerate}[nosep]
  \item elliptic: $q_i(x,y) = \frac{1}{2}(x_i^2+y_i^2)$,
  \item hyperbolic: $q_i (x,y)= x_iy_i$,
  \item focus-focus: $\begin{cases}  q_i (x,y)  = x_iy_{i+1}-x_{i+1}y_i,\\ q_{i+1}(x,y)  = x_iy_i+x_{i+1}y_{i+1},\end{cases}$
  \item non-singular: $q_i (x,y)= y_i$,
 \end{enumerate}
 such that $\{f_i, q_j\}=0$ for all $i,j$.
\end{theorem}

The classification of a non-degenerate singular point can be detected
by computing the eigenvalues of any associated regular element.

\begin{proposition}[\vungoc~{\cite[Chapter 3]{San-book}}]\label{prop_eigenvalues}
If $A$ is a regular element in the Cartan subalgebra generated by the Hessians of the components of
the momentum map (i.e., $A$ has $2n$ distinct eigenvalues) at a fixed point then the eigenvalues of $A$ come in three distinct types of groups:\\[-12pt]
\begin{enumerate}[nosep]
\item a pair of imaginary roots $\pm \mathrm{i}\beta$, called an elliptic block,
\item a pair of real roots $\pm \alpha$, called a hyperbolic block,
\item a quadruple of complex roots $\pm \alpha \pm \mathrm{i}\beta$, called a focus-focus block,\\[-12pt]
\end{enumerate}
where $\alpha,\beta\in\R$.
\end{proposition}

The types of the groups of eigenvalues of $A$ 
agree with the classification of the Cartan subalgebra in \refthmEliasson.
Thus they do not depend on the choice
of the regular element $A$, they only depend on the Cartan subalgebra.

\subsubsection{Degenerate points}

Changing the integrable system on a fixed symplectic manifold
cannot cause a rank 0 point to transition from being focus-focus type
to being elliptic-elliptic type without passing through a degeneracy.

\begin{proposition}\label{prop_degenerate}
 Fix a $4$\--dimensional symplectic manifold $(M,\om)$.
 Let $t_0\in\R$ and let $J_t, H_t\colon M\to \R$ be smooth functions which depend
 smoothly on $t\in\R$.
 Suppose $(J_t,H_t)$ is an integrable system for all $t\in\R$ in an open interval around $t_0$
 and $p\in M$ is a rank 0 fixed point of $(J_t,H_t)$ for all $t\in\R$, which is of type elliptic-elliptic
 for $t>t_0$ and type focus-focus for $t<t_0$.
 Then $(J_{t_0},H_{t_0})$ has a degenerate fixed point at $p$.
\end{proposition}

\begin{proof}
Suppose that $p$ is a non-degenerate fixed point of $(J_{t_0},H_{t_0})$.
Then there exists some
$\gamma,\delta\in\R$ such that
$\omega^{-1}(\gamma d^2 H_{t_0}+\delta d^2 J_{t_0})$ has four distinct eigenvalues
at $p$.
Fix such $\gamma$ and $\delta$.
Since $\gamma d^2 H_{t}+\delta d^2 J_{t}$ is symmetric we see that the characteristic
polynomial of 
$\omega^{-1}(\gamma d^2 H_{t_0}+\delta d^2 J)$ is a constant multiple of a polynomial of the form
\[
 g_t (X) = X^4 + b_t X^2 + c_t
\]
where $b_t,c_t\in\R$ depend continuously on $t$. The zeros of $g_t$ are given by $\pm \sqrt{\kappa_\pm}$ where
\[
\kappa_\pm = \frac{-b_t\pm\sqrt{b_t^2-4c_t}}{2}
\]
and since there are four distinct eigenvalues when
$t=t_0$ we see that 
$b_{t_0}^2-4c_{t_0}\neq 0$.
Thus we see that $g_t$ has four distinct eigenvalues for all
$t$ in a neighborhood of $t_0$.
Since the Williamson type of a fixed point does not depend on the choice
of linear combination as long as one with four distinct eigenvalues is
chosen we see that $g_t$ has zeros of the form $\pm \mathrm{i}\alpha$, $\pm \mathrm{i}\beta$ for $t>t_0$ which means that $b_t^2-4 c_t>0$.
Similarly, we see that $g_t$ has zeros of the form $\alpha\pm\mathrm{i}\beta$
for $t<t_0$ which means that $b_t^2-4 c_t<0$.
Thus, since $b_t-4c_t$ varies continuously with $t$, we see that $b_{t_0}-4c_{t_0} =0$ contradicting our original claim.
\end{proof}

Similar arguments to the one in the proof of Proposition~\ref{prop_degenerate} are used
in Dullin-Pelayo~\cite{dullin-pelayo} and in particular in Figure 4 in that paper.

\begin{remark}\label{rmk_HH}
When a point changes between being of elliptic-elliptic and focus-focus
type it is undergoing what is known as the \emph{Hamiltonian-Hopf bifurcation}, see~\cite{Montaldi_notes}.
\end{remark}

We are grateful to Heinz Han\ss mann and James Montaldi for bringing to our attention the Hamiltonian-Hopf
bifurcation and informing us that our system is undergoing this transformation.

\subsection{Semitoric systems}
\label{sec_semitoricsystems}

Pelayo $\&$ \vungoc~\cite{PVNinventiones, PVNacta} extend the Delzant
classification of toric integrable systems by introducing and classifying what are now known
as semitoric systems.

\begin{definition}
 A \emph{semitoric system} is a integrable
 system of dimension four $(M,\om,(J,H))$ such that:
 \begin{enumerate}
  \item $J$ is proper,
  \item \label{item_2piperiod} the Hamiltonian flow of $J$ (i.e.~the flow of $X^{J}$) is periodic,
  \item all singular points of $(J,H)$ are non-degenerate and have no hyperbolic blocks.
 \end{enumerate}
 A semitoric system is \emph{simple} if there is at most one critical point in $J^{-1}(x)$
 for all $x\in\R$.
\end{definition}
\centerline{\bf Every semitoric system we consider in}
\centerline{\bf this article is a simple semitoric system.}
\smallskip
\noindent Note that $J$ is automatically proper in the case that $M$ is compact.
Concerning item~\eqref{item_2piperiod}, we may assume that $2\pi$ is the minimal period. 
Note that this means the flow of $X^J$ generates a faithful action of $\mbS^1=\R/2\pi\Z$.


If $(M,\om,(J,H))$ is a semitoric integrable system
and $p\in M$ is a rank zero singular point then there are exactly two possibilities for $p$: either $p$ is elliptic\--elliptic or focus\--focus.
Thus, if $A$ is a regular element in the associated Cartan subalgebra then
the eigenvalues of $A$ must either come in two pairs $\pm \mathrm{i}\alpha$, $\pm \mathrm{i}\beta$
in which case $p$ is elliptic-elliptic or come in one quadruple $\pm \alpha\pm\mathrm{i}\beta$ in which case
$p$ is focus-focus, where $\alpha,\beta\in\R$ in each case.  If $p$ is non-degenerate of rank 1
then it must be of elliptic type.

The Pelayo\--V\~{u} Ng\d{o}c classification of simple semitoric integrable systems is in terms of 
five invariants, which we briefly describe now:
\begin{enumerate}
 \item the \emph{number of focus-focus points invariant}: $m_f\in\Z_{\geq 0}$ denotes the number of focus-focus singular points (which is finite by \vungoc~\cite{VuNgoc07}),
 \item the \emph{semitoric polygon}: a family of polygons (analogous to the Delzant polygon of a toric integrable system) which
  encode information about the integral-affine structure of the system. Each element is the image of a toric momentum
  map defined on all of $M$ except certain subsets (which are the union of submanifolds of dimension at most three) related to the focus-focus points,
 \item the \emph{Taylor series invariant}: a Taylor series in two variables for each focus-focus point, which encodes the 
 dynamics of the flow of the Hamiltonian vector fields as they approach the focus-focus fiber
 (originally introduced and described in \vungoc~\cite{VuNgoc03}),
 \item the \emph{volume} or \emph{height invariant}: a real number for each focus-focus point which encodes the height of the focus-focus value in semitoric polygon,
 \item the \emph{twisting index}: an integer assigned to each focus-focus point for each element of the semitoric polygon, which encodes
  the relationship between the toric momentum map used to produce the element of the semitoric polygon and a preferred
  local momentum map around the focus-focus point.
\end{enumerate}
An abstract list of such datas is known as a \emph{list of semitoric ingredients}.
Given semitoric systems $(M_i,\om_i,(J_i,H_i))$ for $i=1, 2$ an \emph{isomorphism of semitoric systems} is a symplectomorphism $\Phi\colon M_1\to M_2$ such that
$\Phi^*(J_2,H_2) = (J_1, f(J_1,H_1))$ where $f\colon\R^2\to\R$ is a smooth function and $\del_y f>0$ everywhere.

\begin{theorem}[Classification by Pelayo $\&$ V\~{u} Ng\d{o}c~\cite{PVNinventiones, PVNacta}]
 The following hold:
 \begin{enumerate}
  \item Two simple semitoric systems are isomorphic if and only if they have the same five semitoric invariants,
  \item Given a list of semitoric ingredients there exists a simple semitoric system which has
   those as its five invariants.
 \end{enumerate}
\end{theorem}

For standard examples of semitoric systems, we refer to Section \ref{examples}.


\subsection{The symplectic structure on $\mbS^2$ and $\mbS^2 \times \mbS^2$}

In order to avoid, on the one hand, confusion concerning the various conventions 
in the literature and, on the other hand, to provide a precise and complete reference,
the following calculations are provided in full.

Let $\mbS^2$ be the unit sphere in $\R^3$ centered at the origin, and let
$(x_1, y_1, z_1, x_2, y_2, z_2)$ be Cartesian coordinates on $\R^3\times\R^3$.
We consider the 4-dimensional manifold $\mbS^2 \x \mbS^2\subset\R^3\times\R^3$ with symplectic form 
$$\om:= \om^{R_1 R_2}:=- (R_1 \om_{\mbS^2} \oplus R_2\om_{\mbS^2})$$ 
where $R_1, R_2 \in \R^{>0}$ and $\om_{\mbS^2}$ is the standard symplectic form on $\mbS^2$. 
Geometrically, the symplectic form $\om_{\mbS^2}$ on $\mbS^2$ is given in $p \in \mbS^2$ by
\begin{align*}
(\om_{\mbS^2})_p(u,v) = \langle p, u \x v \rangle
\end{align*}
where $\langle \cdot, \cdot \rangle$ is the Euclidean scalar product in $\R^3$, $p=(p_1, p_2, p_3) \in \mbS^2$ the 
basepoint and $u=(u_1, u_2, u_3), v=(v_1, v_2, v_3) \in T_p\mbS^2$ tangent 
vectors, i.e., $\langle p, u \rangle =0= \langle p, v \rangle$. To express $\om_{\mbS^2}$ in Cartesian coordinates, we calculate
\begin{align*}
 \langle p, u \x v \rangle & = p_1 \det \begin{pmatrix} u_2 & v_2 \\ u_3 & v_3 \end{pmatrix} + p_2 \det \begin{pmatrix} u_3 & v_3 \\ u_1 & v_1 \end{pmatrix} + p_3 \det \begin{pmatrix} u_1 & v_1 \\ u_2 & v_2 \end{pmatrix} \\
 & = p_1 (dy \wedge dz) (u,v) + p_2 (d z \wedge dx) (u,v) + p_3 (dx \wedge dy) (u,v) 
\end{align*}
and thus
$$
\om_{\mbS^2} = x dy \wedge dz + y dz \wedge dx + z dx \wedge dy.
$$
This implies
$$
\om = - \sum_{i=1}^2 R_i x_i dy_i \wedge dz_i + R_iy_i dz_i \wedge dx_i + R_i z_i dx_i \wedge dy_i
$$
in Cartesian coordinates on $\mbS^2 \x \mbS^2 \subset \R^3 \x \R^3$. We want to use charts on $\mbS^2$ that parametrise the upper and lower hemisphere as graphs over the 2-dimensional unit disk $\D^2$. To keep track of signs, we use $e \in \{+1, -1\}$ in the charts and have
$
\phi_e: \D^2 \to \mbS^2$ with
\[\phi_e(x,y):= \left(x,y, z_e(x,y) \right):= \left(x,y, e \sqrt{1- x^2 -y^2}\right)\]
such that $\phi_{+1}$ covers the northern hemisphere and $\phi_{-1}$ the southern one. Denoting the north and south pole of $\mbS^2$ by $N$ and $S$, we get charts for the `double hemispheres' around $(N,N),  (N,S), (S,N), (S,S) \in \mbS^2 \x \mbS^2$ via choosing $e_1, e_2 \in \{+1, -1\}$ accordingly and setting
\begin{align}
\label{chart}
 \phi_{e_1, e_2} &:  \D^2 \x \D^2 \to \mbS^2 \x \mbS^2, \\ \notag
 \phi_{e_1 e_2}(x_1, y_1, x_2, y_2) 
 & :=\left(x_1, y_1, e_1 \sqrt{1 - x_1^2 - y_1^2}, x_2, y_2, e_2 \sqrt{1-x_2^2 - y_2^2}\ \right).
\end{align}
For better readability, let us drop the subscripts $e$, $e_1$, and $e_2$ whenever the context allows, and
introduce a function $z_i(x_i,y_i)$, i.e., we write
$$
\phi= \phi_e, \quad z(x,y)= e\sqrt{1-x^2-y^2}, \quad \phi= \phi_{e_1e_2}, \quad  z_i(x_i, y_i):= e_i\sqrt{1-x_i^2-y_i^2}
$$
whenever possible.
Now we express $\om_{\mbS^2}$ in the new coordinates $\phi_e$.
We compute
\begin{align}
 \del_x z(x,y) & = \frac{- e  x}{\sqrt{1-x^2 - y^2}} = \frac{-x}{z(x,y)} \label{eqn:dxz}\\
  \del_y z(x,y) & = \frac{- e  y}{\sqrt{1-x^2 - y^2}} = \frac{-y}{z(x,y)} \label{eqn:dyz}
\end{align}
yielding
\begin{align}
\label{dz}
 d ( z(x,y)) & =  \frac{-x}{z(x,y)} dx +  \frac{-y}{z(x,y)}  dy 
\end{align}
leading to
\begin{align*}
 \phi^*\om_{\mbS^2} &  = \left( \frac{x^2}{z(x,y)} + \frac{y^2}{z(x,y)} + z(x,y) \right)dx \wedge dy =  \frac{1}{z(x,y)} dx \wedge dy .
\end{align*}
Subsequently we get for $\om$ in coordinates $\phi= \phi_{e_1 e_2}$ 
the expression
\begin{align*}
 \phi^*\om & =- \ \phi^*(R_1 \om_{\mbS^2} \oplus R_2 \om_{\mbS^2}) \\
 & = - \left( \frac{ R_1}{z_1(x_1, y_1)} \ dx_1 \wedge dy_1 + \frac{R_2}{z_2(x_2, y_2)} \ dx_2 \wedge dy_2\right)
\end{align*}
and thus in matrix form we have
\begin{align}
\label{omegaInverse}
\om = 
\begin{pmatrix}
 0 &  \frac{-  R_1}{z_1} & 0 & 0 \\
  \frac{  R_1}{z_1} & 0 & 0 &0 \\
  0 & 0 &  0 & \frac{-  R_2}{z_2} \\
  0 & 0 &  \frac{  R_2}{z_2} & 0 
\end{pmatrix}
\qquad\textrm{and}\qquad
\om^{-1} = 
\begin{pmatrix}
 0 &  \frac{z_1}{  R_1} & 0 & 0 \\
  \frac{-z_1}{ R_1} & 0 & 0 &0 \\
  0 & 0 &  0 & \frac{z_2}{ R_2} \\
  0 & 0 &  \frac{-z_2}{ R_2} & 0 
\end{pmatrix}.
\end{align}

Suppose $f: \mbS^2 \x \mbS^2 \to \R$.
Using the charts $\phi_{e_1 e_2}$, we compute for $h := f \circ \phi_{e_1 e_2}: \D \x \D \to \R$ the differential
$$
dh = \sum_{i=1}^2 \del_{x_i}h \ dx_i + \del_{y_i} h \ dy_i
$$
and can solve $\om(X^h, \cdot) = -dh$ for $X^h$ via
$$
(X^h)^T = -\begin{pmatrix} \del_{x_1}h ,\ \del_{y_1}h, \del_{x_2}h ,\ \del_{y_2}h \end{pmatrix} \ \om^{-1}
$$
so
\begin{equation}
 \label{HamVF}
X^h(x_1, y_1, x_2, y_2) = 
\begin{pmatrix}
 \frac{\del_{y_1} h(x_1, y_1, x_2, y_2) \ z_1(x_1, x_2)}{ R_1} \\
 - \frac{\del_{x_1} h(x_1, y_1, x_2, y_2) \  z_1(x_1, x_2)}{  R_1} \\
 \frac{\del_{y_2}h (x_1, y_1, x_2, y_2) \ z_2(x_2, y_2)}{ R_2} \\
 - \frac{\del_{x_2}h (x_1, y_1, x_2, y_2) \ z_2(x_2, y_2)}{  R_2} 
\end{pmatrix}.
\end{equation}


\subsection{Explicit examples of semitoric systems}
\label{examples}

Consider the manifold $\mbS^2\x\mbS^2$ with symplectic form $\om: = -(R_1\om_{\mbS^2} \oplus R_2\om_{\mbS^2})$ 
where $\om_{\mbS^2}$ is the standard volume form on $\mbS^2$ 
and $0<R_1<R_2$ are real numbers.

\begin{example}[{\bf Coupled angular momenta}]
\label{coupledAngular}
The coupled angular momenta system
is given by $J_{\vec{R}},H_t\colon\mbS^2\x\mbS^2\to\R$ with
 \begin{equation}\label{eqn_system}
  \begin{cases}
   J_{\vec{R}} (x_1, y_1, z_1, x_2, y_2, z_2) : = R_1 z_1 + R_2 z_2,\\
   H_t(x_1, y_1, z_1, x_2, y_2, z_2) : = (1-t) z_1 + t (x_1 x_2 + y_1 y_2 + z_1 z_2)
  \end{cases}
 \end{equation}
where $(x_i, y_i, z_i)$ are Cartesian coordinates on $\mbS^2 \subset\R^3$ for $i=1,2$, 
$t\in [0,1]$ is the \emph{coupling parameter}, and $\vec{R}=(R_1,R_2)\in\R^2$
with $0<R_1<R_2$.
\end{example}

This system was originally introduced in Sadovski\'\i\ $\&$ Z\^hilinski\'\i~\cite{SaZh-PhysLettersA} and studied in detail
in Le Floch $\&$ Pelayo~\cite{LFPecoupledangular}, where it is shown that
there exist two fixed values $t_-, t_+\in (0,1)$ with $t_-<t_+$ which depend on $R_1, R_2$ such that

\begin{enumerate}
 \item if $t_-<t<t_+$ then $(J_{\vec{R}},H_t)$ is a semitoric system with exactly one focus-focus point,
 \item if $t>t_+$ or $t<t_-$ the $(J_{\vec{R}},H_t)$ is a semitoric system with exactly zero focus-focus points (these
  are known as systems of \emph{toric type}, see Section 2 of \vungoc~\cite{VuNgoc07}),
 \item if $t=t_-$ or $t=t_+$ then $(J_{\vec{R}},H_t)$ has a degenerate singular point, and thus is not
  a semitoric system.
\end{enumerate}

The image of the momentum map for \refcoupledAngular\ with varying values of $t$ is shown in Figure~\ref{fig_coupledangular}.

\begin{figure}[h]
\includegraphics[width=350pt]{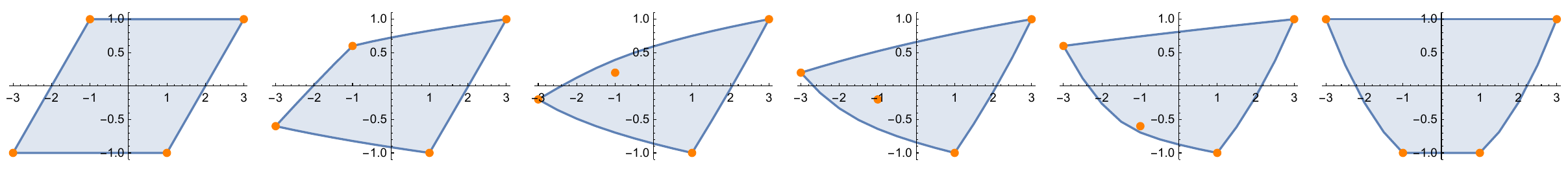}
\caption{The momentum map image for the coupled angular momenta system with the rank zero points marked.  As 
the coupling parameter $t$ changes one of the rank zero points transitions from being elliptic-elliptic to being focus-focus and then
back to elliptic-elliptic.}
\label{fig_coupledangular}
\end{figure}

Another standard example of a semitoric system is 

\begin{example}[{\bf Coupled spin oscillator}]
\label{coupledSpin}
The coupled spin oscillator system is given by $J, H: \mbS^2 \times \R^2 \to \R$ where
\begin{align*}
 J (x, y, z, u,v) := \frac{1}{2}(u^2+v^2) + z \quad \mbox{and} \quad  H(x, y, z, u,v) & := \frac{1}{2}(ux + vy)
\end{align*}
with Cartesian coordinates $(x,y,z)$ on $\mbS^2$ and $(u,v)$ on $\R^2$. 
\end{example}

See Pelayo $\&$ \vungoc~\cite{PeVNcomm-mathphys} for a detailed investigation of \refcoupledSpin.

\begin{remark}
 The spherical pendulum consists of $J, H : \mathrm{T}^*\mbS^2 \to \R$ with
 \begin{align*}
  J(q_1, q_2, q_3, p_1, p_2, p_3) & :=q_1 p_2 - q_2 p_1, \\
  H(q_1, q_2, q_3, p_1, p_2, p_3) & :=\frac{1}{2}(p_1^2+p_2^2+p_3^2)+q_3
 \end{align*}
 and satisfies nearly all of the requirements to be semitoric, but
 the $J$ is not proper since the momentum map image contains unbounded vertical lines.
 However, the spherical pendulum is a so-called \emph{generalized semitoric system}, as discussed
 in Pelayo $\&$ Ratiu $\&$ \vungoc~\cite{PeRaVN2015}.
 For this same reason, the quadratic spherical pendulum 
 (see for example Cushman $\&$ \vungoc{}~\cite{CuVN2002} and Efstathiou $\&$ Martynchuk~\cite{efstathiou})
 is not a semitoric integrable system.
\end{remark}


\section{A family of systems with two focus-focus points}
\label{sec_familyofsystems}

In this section we introduce the system which is the subject of this
paper and prove \refthmIntro, our main result.
This system is \emph{minimal} in the sense of Kane $\&$ Palmer $\&$ Pelayo~\cite{KaPaPeminimal}, i.e., it is not possible to perform a \emph{blowdown of toric type} on the system (see Kane $\&$ Palmer $\&$ Pelayo~\cite[Section 4.1]{KaPaPeminimal}
for a description of this operation).
Minimal semitoric integrable systems are classified
in Kane $\&$ Palmer $\&$ Pelayo~\cite{KaPaPeminimal} and the system discussed in the present
paper is minimal of type (2), using the terminology of that paper.

\subsection{The system}

Consider $R_1, R_2 \in \R^{>0}$ as scaling of radii with $R_1<R_2$ and endow $\mbS^2 \x \mbS^2$ with the 
symplectic form $\om=\om^{R_1 R_2}$. Let $\vec{t}:=(t_1, t_2, t_3, t_4) \in \R^4$ be parameters,
let $\vec{R}=(R_1,R_2)$, 
and define $\Phi:=(J_{\vec{R}}, H_{\vec{t}}): \mbS^2 \x \mbS^2 \to \R^2$ in Cartesian coordinates by
\begin{equation*}
 \left\{
  \begin{aligned}
   J_{\vec{R}}(x_1, y_1, z_1 , x_2, y_2, z_2) & := R_1 z_1 + R_2 z_2 , \\
   H_{\vec{t}}(x_1, y_1, z_1 , x_2, y_2, z_2) & := t_1 z_1 + t_2 z_2 + t_3 (x_1 x_2 + y_1 y_2) + t_4 z_1 z_2
  \end{aligned}
 \right.
\end{equation*}
as in Equation~\eqref{eqn_thesystem}.

Unless we explicitly need the parameters we often write $J := J_{\vec{R}}$ and $H:= H_{\vec{t}}$ for brevity.
The main result of this section is

\begin{theorem}
\label{mainThm}
The following hold:
\begin{enumerate}
  \item [1)]
  The system \eqref{eqn_thesystem} is a compact integrable system for all choices of parameters with $t_3\neq 0$,
  \item [2)]
  The system \eqref{eqn_thesystem} is semitoric and has two focus-focus points for parameters in a neighborhood of
  $$
  R_1=1,\,\, R_2=2,\,\, t_1 = \frac{1}{4},\,\, t_2 = \frac{1}{4},\,\, t_3 = \frac{1}{2},\,\, t_4=0.
  $$
 \end{enumerate}
\end{theorem}

Theorem~\ref{mainThm} is a combination
of Propositions~\ref{prop_2FFpoints-nbhd} and ~\ref{prop_integrable} and Corollary~\ref{cor_rank1nondegents} which 
we prove in the remainder of this section.

\begin{remark}
 At the parameters for which the system in Equation~\eqref{eqn_thesystem} is has two focus-focus
 points it enjoys a certain sense of uniqueness.  As shown in~\cite[Theorem 2.5]{KaPaPeminimal},
 up to scaling the lengths of the sides, there is only one semitoric polygon for which the corresponding
 system is compact with two focus-focus points such that $J$ has isolated fixed points.
 Thus, this semitoric polygon is the one associated to the system in Equation~\eqref{eqn_thesystem}.
 By evaluating $J$ on the rank zero points (see Lemma~\ref{rankNull}) we can easily find the semitoric polygon
 for the system~\eqref{eqn_thesystem}, as shown in Figure~\ref{fig_semitoricpoly}.
\end{remark}

\begin{figure}[h]
\includegraphics[width=350pt]{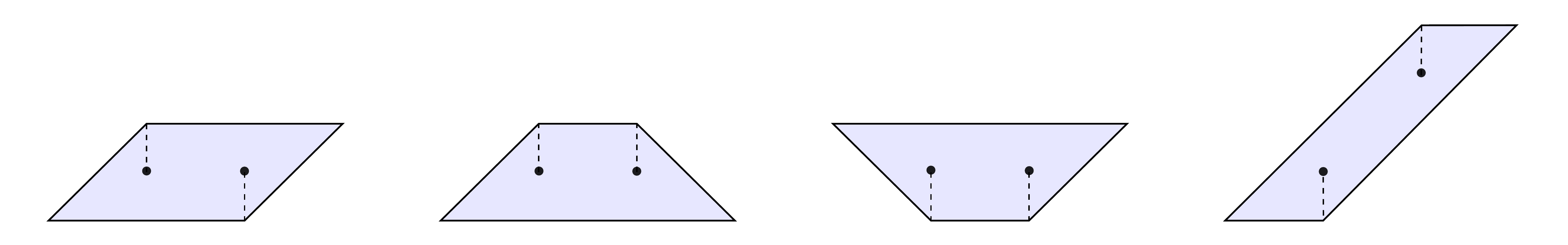}
\caption{Four semitoric polygons associated to the system~\eqref{eqn_thesystem}.
The slanted edges all have slope $\pm 1$.
For each polygon the $x$\--coordinates of the vertices, from left to right, 
are $-R_1-R_2$, $R_1-R_2$, $-R_1+R_2$, and $R_1+R_2$ (since we assume $R_1<R_2$).}
\label{fig_semitoricpoly}
\end{figure}

\begin{remark}
 At first, we considered the system
 \begin{equation*}\label{eqn_lsystem}
\begin{aligned}
 J_{\vec{R}}(x_1, y_1, z_1, x_2, y_2, z_2)&   := R_1 z_1 + R_2 z_2 , \\
H_{(\ell_1,\ell_2)}(x_1, y_1, z_1, x_2, y_2, z_2) &  := (1-\ell_1) z_1 + (1- \ell_2) z_2 + \ell_1 \ell_2 (x_1x_2 + y_1 y_2 + z_1 z_2)
\end{aligned}
\end{equation*}
 for $\ell_1,\ell_2\in[0,1]$
 hoping to generalize the construction of the coupled angular momentum,
 but numerical evidence strongly suggests 
 that while there are two different points that become focus-focus for certain values
 there is no choice of $\ell_1$ and $\ell_2$ for which both points are
 focus-focus simultaneously.
\end{remark}


\subsection{Rank 0 points and their nondegeneracy}

In the chart $\phi$, the integrals $J$ and $H$ are given by
\begin{align}
\label{Jchart}
 J(x_1, y_1, x_2, y_2) & =  R_1 z_1 + R_2 z_2 , \\
 \label{Hchart}
 H(x_1, y_1, x_2, y_2) & = t_1 z_1 + t_2 z_2 + t_3 (x_1 x_2 + y_1 y_2) + t_4 z_1 z_2. 
\end{align}
where each $z_i = z_i(x_i,y_i)$ is a function of $x_i$ and $y_i$ for $i=1,2$. 
Using equations \eqref{eqn:dxz}, \eqref{eqn:dyz}, and \eqref{HamVF}, the Hamiltonian vector fields are given by
\begin{equation}
 \label{HamJH}
X^J(x_1, y_1, x_2, y_2) =  
\begin{pmatrix}
 -y_1 \\
   x_1 \\
 -y_2\\
  x_2
\end{pmatrix}
,
X^H(x_1, y_1, x_2, y_2) =
\begin{pmatrix}
 \,\frac{-t_1 y_1 + t_3 y_2 z_1 - t_4 z_2 y_1}{R_1} \\[6pt]
 \frac{t_1 x_1 - t_3 x_2 z_1 + t_4 z_2 x_1}{  R_1} \\[6pt]
 \frac{ -t_2 y_2  + t_3 y_1 z_2 - t_4 z_1 y_2 }{ R_2} \\[6pt]
 \frac{ t_2 x_2  - t_3 x_1 z_2 + t_4 z_1 x_2 }{  R_2} 
\end{pmatrix}.
\end{equation}


\noindent
Recall that $N$ denotes the north pole of $\mbS^2$ and $S$ the south pole.

\begin{lemma}
\label{rankNull}
 The set of rank 0 points of $(J,H)$, i.e., the set of fixed points, is given by $\{(N,N), (N,S), (S,N), (S,S)\}$.
\end{lemma}

\begin{proof}
 Geometrically, $J$ is the sum of the height function
 on each factor of the product $\mbS^2 \x \mbS^2$ 
 scaled by $R_1$ and $R_2$ respectively. 
Thus $J$ gives rise to horizontal rotations on each of the two spheres 
and its Hamiltonian flow has fixed points exactly at $\{(N,N), (N,S), (S,N), (S,S)\}$. 
The function $J$ reaches its global maximum, $R_1+ R_2$, at $(N,N)$ and its global 
minimum, $-(R_1 + R_2)$, at $(S,S)$. The corresponding fibers $J^{-1}(R_1 + R_2)$ 
and $J^{-1}(-(R_1 + R_2))$ consist exactly of the singletons $ \{(N,N)\}$ and $\{(S,S)\}$. 

Fixed points of $(J,H): \mbS^2 \x \mbS^2 \to \R$ require $\rk D(J,H) = 0$. 
Therefore they must have $DJ=0$ which is equivalent to $X^J=0$, i.e., when we look for 
fixed points of $(J,H)$, the only candidates are the points $(N,N)$, $(N,S)$, $(S,N)$, 
and $(S,S)$ for which we have to check if additionally $DH=0$ or equivalently $X^H=0$ holds.

Since all possible fixed points lie in the range of the charts $\phi_{e_1e_2}$ we can check 
the values of $X^H$ by using formula \eqref{eqn:dxz}, \eqref{eqn:dyz}, and \eqref{HamJH}. The 
corresponding point in the domain is in all cases $(x_1, y_1, x_2, y_2)= (0,0,0,0)$ and 
we compute that $X^H(0,0,0,0)$ vanishes
and thus $\{(N,N), (N,S), (S,N), (S,S)\}$ is indeed the fixed point set of $(J,H)$.
\end{proof}

Keep in mind from the above proof that the rank 0 points correspond to the origin in the charts in \eqref{chart}. 

\begin{lemma}
 At the origin $p= (0,0,0,0)$ in the charts in \eqref{chart}, we find
 \begin{equation}
  \om^{-1}_p d^2 J(p) = 
  \begin{pmatrix}
   0 & -1 & 0 & 0 \\
   1 & 0 & 0 & 0 \\
   0 & 0 & 0 & -1 \\
   0 & 0 & 1 & 0
  \end{pmatrix}
 \end{equation}
and 
\begin{equation}
\label{omHessH}
  \om^{-1}_p d^2 H(p) = 
  \begin{pmatrix}
 0 & - \frac{t_1 +  e_2 t_4}{R_1} & 0 & \frac{e_1 t_3 }{R_1} \\
  \frac{t_1 + e_2 t_4}{R_1} & 0 & - \frac{ e_1 t_3}{R_1} & 0 \\
    0 & \frac{ e_2 t_3 }{R_2} & 0 & - \frac{t_2 + e_1 t_4}{R_2} \\
 - \frac{ e_2 t_3 }{R_2} & 0 &  \frac{t_2 + e_1 t_4}{R_2} & 0 
 \end{pmatrix}.
 \end{equation}
\end{lemma}

\begin{proof}
We compute the Hessians of $J$ and $H$ using \eqref{Jchart} and \eqref{Hchart}. 
 Since derivatives are additive we can first calculate the Hessians of their 
 components seperately. Recall from \eqref{eqn:dxz} and \eqref{eqn:dyz} 
 that $\del_{x_i}z_i = \frac{-x_i}{z_i}$ and $\del_{y_i}z_i = \frac{-y_i}{z_i}$, yielding
 \begin{align*}
  \del^2_{x_i x_i}z_i = \frac{- z_i + x_i \del_{x_i}z_i}{z_i^2}, \quad \del^2_{x_i y_i} z_i = \frac{ x_i \del_{y_i}z_i}{z_i^2}, \quad\textrm{and}\quad \del^2_{y_i y_i}z_i = \frac{- z_i + y_i \del_{y_i}z_i}{z_i^2}.
 \end{align*}
Since $z_1$ does not depend on $x_2, y_2$ and $z_2$ does not depend on $x_1$ and $y_1$ we obtain for the Hessian of $z_i$ w.r.t.\ the variables $x_1$, $y_1$, $x_2$, $y_2$ in $p$
\begin{align*}
 d^2z_1(p) = 
 \begin{pmatrix}
  -e_1 & 0 & 0 & 0 \\
  0 & -e_1 & 0 & 0 \\
  0 & 0 & 0 & 0 \\
  0 & 0 & 0 & 0 
 \end{pmatrix}
 \quad \mbox{and} \quad
  d^2z_2(p) = 
 \begin{pmatrix}
  0 & 0 & 0 & 0 \\
  0 & 0 & 0 & 0  \\
  0 & 0 & -e_2 & 0  \\
  0 & 0 & 0 & -e_2 \\
  \end{pmatrix}.
\end{align*}
Next we consider the term $x_1x_2 + y_1 y_2$ and get
\begin{align*}
 d^2(x_1 x_2 + y_1 y_2)(p) = 
 \begin{pmatrix}
  0 & 0 & 1 & 0 \\
  0 & 0 & 0 & 1 \\
  1 & 0 & 0 & 0 \\
  0 & 1 & 0 & 0 
 \end{pmatrix}. 
\end{align*}
For the term $z_1z_2$, we get 
\begin{align*}
 d^2 (z_1 z_2)(p) = 
 \begin{pmatrix}
    -e_1e_2 & 0 & 0 & 0 \\
    0 & -e_1e_2 & 0 & 0 \\
     0 & 0 & -e_1e_2  & 0 \\
   0 &  0 & 0 & -e_1e_2  
 \end{pmatrix}.
\end{align*}
The equations \eqref{Jchart} and \eqref{Hchart} together with the above calculations yield
\begin{align*}
 d^2 J(p) & =  
 \begin{pmatrix}
  -e_1R_1 & 0 & 0 & 0 \\
  0 & -e_1R_1 & 0 & 0 \\
  0 & 0 & -e_2R_2 & 0 \\
  0 & 0 & 0 & -e_2 R_2 
 \end{pmatrix}
\end{align*}
and
\begin{align*}
 d^2 H(p) & = 
 \begin{pmatrix}
  - t_1 e_1 - t_4 e_1 e_2 & 0 & t_3 & 0 \\
  0 & - t_1 e_1 - t_4 e_1 e_2 & 0 & t_3 \\
  t_3 & 0 & - t_2 e_2 - t_4 e_1 e_2 & 0 \\
  0 & t_3 & 0 & - t_2 e_2 - t_4 e_1 e_2 
 \end{pmatrix}.
\end{align*}
Evaluating $\om^{-1}$ from \eqref{omegaInverse} at $p$ yields
\begin{equation*}
\om^{-1}_p = 
 \begin{pmatrix}
  0 & \frac{e_1}{R_1} & 0 & 0 \\
  -\frac{e_1}{R_1}  & 0 & 0 & 0 \\
  0 & 0 & 0 &  \frac{e_2}{R_2} \\
  0 & 0 &  - \frac{e_2}{R_2} & 0
 \end{pmatrix}
\end{equation*}
and therefore, using $e_1^2 = 1 = e_2^2$, we get the desired results for $\om_p^{-1}J(p) $ and $\om_p^{-1}H(p)$.
\end{proof}

Given a polynomial of the form $ay^2 + by + c$, the expression $b^2 -4ac$ is called
 the {\em discriminant} of the polynomial. Thus, a straightforward calculation yields

\begin{corollary}
 \label{discrim}
 Denote by $I$ the $4 \x 4$ identity matrix. Then the characteristic polynomial of $\om^{-1}_p d^2 H(p)$ is given by
 \begin{align*}
  \chi(X)& := \det\bigl( \om^{-1}_p d^2 H(p) - X I \bigr) \\
 & =  X^4 + \frac{1}{R	_1^2 R_2^2} \left( R_1^2 ( t_2 + e_1 t_4)^2 + 2 e_1 e_2 R_1 R_2 t_3^2 + R_2^2(t_1 + e_2 t_4)^2 \right) X^2 \\
 & \quad + \frac{1}{R_1^2 R_2^2}\left( ( t_2 + e_1 t_4)^ 2 (t_1 + e_2 t_4)^2  - 2 e_1 e_2( t_2 + e_1 t_4)(t_1 + e_2 t_4) t_3^2 + t_3^4  \right)
 \end{align*}
which is a polynomial of second degree in $Y:=X^2$ with discriminant
\begin{align}\label{eqn_discrimiant}
 \lapla  :=& \lapla_{\vec{R}, \vec{t}, e_1, e_2}\\ \notag
 =&  \left(\frac{1}{ R_1^2 R_2^2} \left( R_1^2 ( t_2 + e_1 t_4)^2 + 2 e_1 e_2 R_1 R_2 t_3^2 + R_2^2(t_1 + e_2 t_4)^2 \right)\right)^2 \\ \notag
 & \quad - \frac{4}{R_1^2 R_2^2}\left( ( t_2 + e_1 t_4)^ 2 (t_1 + e_2 t_4)^2  - 2 e_1 e_2( t_2 + e_1 t_4)(t_1 + e_2 t_4) t_3^2 + t_3^4  \right).
\end{align}
\end{corollary}

Now we want to determine the type of the rank $0$ points located at $(N, N)$, $(N, S)$, $(S, N)$, $(S, S)$, i.e., if they are nondegenerate or not and, in case they are nondegenerate, if they are focus-focus or elliptic-elliptic or something else. We will see that the type of the rank $0$ points highly depends on the choice of parameters $\vec{R} $ and $\vec{t}$. 
\begin{proposition}[{\bf Rank 0 Criterion}]\label{prop_rank0criterion}
 Suppose $p\in \mathbb{S}^2\times \mathbb{S}^2$ has $z$-coordinates $(e_1, e_2)\in \{-1,1\}^2$.  Then $p$ is a rank
 0 singular point of $(J_{\vec{R}},H_{\vec{t}})$.  If $\lapla_{\vec{R}, \vec{t}, e_1, e_2}<0$ then $p$ is non-degenerate of
 focus-focus type, and
 if $\lapla_{\vec{R}, \vec{t}, e_1, e_2}>0$ then $p$ is non-degenerate and is of type elliptic-elliptic,
 elliptic-hyperbolic, or hyperbolic-hyperbolic.
\end{proposition}

\begin{proof}
We already know that the set of rank 0 point are exactly those with $z$\--coordinates $\pm 1$ by Lemma~\ref{rankNull}.
Note that the characteristic polynomial of $\om^{-1}_p d^2 H(p)$ has zeros
\[
 X = \pm\sqrt{\frac{-1}{2R_1^2 R_2^2} \left( R_1^2 ( t_2 + e_1 t_4)^2 + 2 e_1 e_2 R_1 R_2 t_3^2 + R_2^2(t_1 + e_2 t_4)^2 \right) \pm \frac{\sqrt{\lapla}}{2}}
\]
where $\lapla:=\lapla_{\vec{R}, \vec{t}, e_1, e_2}$ is as in Equation~\eqref{eqn_discrimiant}.

If $\lapla<0$ then there are four eigenvalues which take the form $\al\pm\mathrm{i}\be$ for $\al,\be\in\R$,
and thus $p$ is focus-focus by Proposition~\ref{prop_eigenvalues}.
If, $\lapla>0$ then $p$ is a non-degenerate fixed point which is not focus-focus, so
it is either elliptic-elliptic, hyperbolic-hyperbolic, or hyperbolic-elliptic.
\end{proof}

Note that in the case $\lapla = 0$ the point can still be non-degenerate, but
Proposition~\ref{prop_rank0criterion} does not give us any information in this case.
The following statement implies that there exist parameter values for which the system has four nondegenerate rank $0$ points, two of them elliptic-elliptic and two focus-focus, and is proved by plugging the values into the criterion in
Proposition~\ref{prop_rank0criterion}.

\begin{corollary}\label{cor_goodparameters}
 For the parameter values
 \begin{equation}
 \label{goodParam}
  R_1 = 1, \quad R_2 = 2, \quad t_1 = \frac{1}{4}, \quad t_2 = \frac{1}{4}, \quad t_3 = \frac{1}{2}, \quad t_4 = 0,
 \end{equation}
the matrix $\om^{-1}_p d^2 H(p)$ has four distinct eigenvalues at each of the points $ p \in \{(N,N), (S,S), (N,S), (S,N)\}$ given by
\begin{align*}
 Eig(N,N) = Eig(S,S) =& \left\{\pm\frac{\mathrm{i}}{8}\sqrt{\frac{21+3\sqrt{33}}{2}},\,\pm\frac{\mathrm{i}}{8}\sqrt{\frac{21-3\sqrt{33}}{2}}\right\} , \\
 Eig(N,S) = Eig(S,N)=& \left\{\left(\sqrt{\frac{5}{32}}\right)\left(\pm\cos\left(\frac{1}{2}\mathrm{arctan}\left(\frac{3\sqrt{31}}{11}	\right)\right)\right.\right.\\
                     &\left.\left. \pm\mathrm{i}\sin\left(\frac{1}{2}\mathrm{arctan}\left(\frac{3\sqrt{31}}{11}\right)\right)\right)\right\},
\end{align*}
and thus 
$p$ is a nondegenerate fixed point according to \refnondegBF. In particular, 
$(N, N)$ and $(S,S)$ are elliptic-elliptic and $(N, S)$ and $(S, N)$ are focus-focus.
\end{corollary}

Since nonvanishing and noncoinciding are open conditions, there exist in fact intervals around the parameters \eqref{goodParam} where the systems continues to have two focus-focus and two elliptic-elliptic points.

\begin{proposition}\label{prop_2FFpoints-nbhd}
 There exists an open set $U\subset \R^6$ which contains the point $(1,2,\frac{1}{4},\frac{1}{4},\frac{1}{2},0)$ such that for all $(R_1,R_2,t_1,t_2,t_3,t_4)\in U$ the system given in Equation~\eqref{eqn_thesystem} has elliptic-elliptic points at $(N, N)$ and $(S,S)$ and focus-focus points at $(N, S)$ and $(S, N)$.
\end{proposition}


\subsection{Rank 1 points} 

We want to study rank 1 points by means of cylindrical coordinates. To avoid the problems with cylindrical coordinates
near poles we state the following observation.

\begin{lemma}
\label{noPole}
 Let $t_1$, $t_2$, $t_3$, $t_4 \in \R$, with $t_3 \neq 0$ and let $R_1$, $R_2 \in \R^{>0}$.
 If $(x_1, y_1, z_1, x_1, y_2, z_2) \in \mbS^2 \x \mbS^2$ is a critical point of rank 1 of \eqref{eqn_thesystem} then $z_1, z_2 \neq \{\pm 1\}$.
\end{lemma}

\begin{proof}
 Critical points of $(J,H)$ from \eqref{eqn_thesystem} are those $p\in\mbS^2\times\mbS^2$ such that
 $dH(p)$ and $dJ(p)$ are linearly dependent, which is equivalent to the existence of
 a nonzero $\lambda\in\R$ such that $d(H-\lambda J)(p)=0$
 since $dJ=0$ only occurs at the rank 0 points.
 Defining
 $f_1, f_2: \R^6 \to \R$ by $f_i(x_1, y_1, z_1, x_2, y_2, z_2) := x_i^2 + y_i^2 +z_i^2$ for $i=1,2$, this is equivalent 
 to looking for critical points of $H - \lam J: \R^6 \to \R$ on the set $f_1^{-1}(1) \cap f^{-1}_2(1)$, 
 i.e., critical points can be computed by means of Lagrangian multipliers, i.e., a critical point 
 $p:=(x_1, y_1, z_1, x_1, y_2, z_2)$ satisfies the equations
 \begin{align*}
 \left\{
\begin{aligned}
& \nabla H(p) = \lam \nabla J(p) + \mu_1 \nabla f_1(p) + \mu_2 \nabla f_2(p) ,   \\
& x_1^2 + y_1^2 + z_1^2 = 1 , \\
& x_2^2 + y_2^2 + z_2^2 = 1 
\end{aligned}
\right.
\end{align*}
for some $\lambda, \mu_1, \mu_2 \in \R$. Using the gradient with respect to the Euclidean metric, we obtain
\begin{align}\label{eqn_lagrangemult}
 \begin{pmatrix}
  t_3 x_2 \\ t_3 y_2 \\ t_1 + t_4 z_2 \\ t_3 x_1 \\ t_3 y_1 \\ t_2 + t_4 z_1
 \end{pmatrix}
=
\begin{pmatrix}
 0 \\ 0 \\ \lam R_1 \\ 0 \\ 0 \\ \lam R_2 
\end{pmatrix}
+
\begin{pmatrix}
 2 \mu_1 x_1 \\ 2 \mu_1 y_1 \\ 2 \mu_1 z_1 \\ 0 \\ 0 \\ 0 
\end{pmatrix}
+
\begin{pmatrix}
0 \\ 0 \\ 0 \\ 2 \mu_2 x_2 \\ 2 \mu_2 y_2 \\ 2 \mu_2 z_2 
\end{pmatrix}.
\end{align}
Recall that the rank 0 points are precisely those with $z_1\in\{\pm 1\}$ and $z_2\in\{\pm 1\}$ simultaneously.
Suppose that $z_1\in\{\pm 1\}$ which implies $x_1 = y_1=0$ since $(x_1,y_2,z_1)\in\mbS^2$.
Then, recalling that $t_3 \neq 0$, we see that Equation~\eqref{eqn_lagrangemult} implies $x_2 = y_2 =0$ which 
in turn implies
$z_2\in\{\pm 1\}$ so the only solution is in fact a rank 0 point.
The same argument works if we assume $z_2\in\{\pm 1\}$.
\end{proof}

We now introduce cylindrical coordinates on $\mbS^2 \x \mbS^2$ via 
\begin{equation*}
 (x_i, y_i, z_i) \mapsto \begin{pmatrix}\sqrt{1-z_i^2} \cos (\theta_i), & \sqrt{1-z_i^2} \sin (\theta_i), & z_i \end{pmatrix}
\end{equation*}
where $ i \in \{1, 2\}$ and $\theta_i$ is the counterclockwise angle between the $x_i$-axis and $(x_i, y_i)$ in $\R^2$. In these coordinates, the system \eqref{eqn_thesystem} becomes
\begin{align*}
 J(\theta_1, z_1, \theta_2, z_2) & = R_1 z_1 + R_2 z_2, \\
 H(\theta_1, z_1, \theta_2, z_2) & = t_1 z_1 + t_2 z_2 + t_3\sqrt{(1-z_1^2)(1-z_2^2)} \cos(\theta_1 - \theta_2) + t_4 z_1 z_2
\end{align*}
and the symplectic form is
\begin{equation}\label{eqn_symplform-cylindrical}
 \om = R_1 dz_1\wedge d\theta_1 +R_2 dz_2\wedge d\theta_2.
\end{equation}
According to \refnoPole\ these coordinates are valid where rank 1 
points may occur (if the rank 1 point occurs at the discontinuity of $\theta_i$, then shift the domain of
definition of $\theta_i$).
We compute the derivative
\begin{align}\label{eqn_dJ}
 d J(\theta_1, z_1, \theta_2, z_2)   = \begin{pmatrix} 0, & R_1 , & 0, & R_2 \end{pmatrix} 
\end{align}
which never vanishes. 
Therefore we have \refmorse\ at our disposal.


Let us compute the symplectic quotient $(\mbS^2 \x \mbS^2) \sslash \mbS^1$ where the $\mbS^1$-action is induced by $J$. 
Given $c \in\ ]-(R_1 + R_2), (R_1 + R_2)[$, which is the set of regular values of $J$, we can solve for $z_1$
on the level set $J^{-1}(c)$ to find
$$
z_1 = \frac{c- R_2 z_2}{R_1}.
$$
By  Equations~\eqref{eqn_dJ} and~\eqref{eqn_symplform-cylindrical}
we see that 
$X^J = \partial_{z_1}+\partial_{z_2}$ so
the flow of $J$ rotates $\theta_1$ and $\theta_2$ by a common
angle. Thus, the $\mbS^1$\--action produced by the flow 
of $X^J$ preserves the angle difference $\theta_1 - \theta_2$. Now 
consider the chart on the quotient $J^{-1}(c) / \mbS^1$ with coordinates $(\ze, \vartheta)$ given by
\begin{align*}
  \ze:= z_1  \quad \mbox{and} \quad \vartheta := \theta_1 - \theta_2
\end{align*}
where
\[
 -1 < \ze < 1\qquad \textrm{and}\qquad
 \frac{c-R_2}{R_1} < \ze < \frac{c+R_2}{R_1}
\]
since $-1<z_1, z_2<1$. All rank 1 critical points occur in this chart since by \refnoPole\ rank 1 points do not occur when $z_1=\pm 1$
or $z_2 = \pm 1$. We now let $H$ descend to the symplectic quotient $(\mbS^2 \x \mbS^2) \sslash \mbS^1$ where it reads
\begin{align*}
 &H(\ze, \vartheta)\\
  & = t_1 \ze + t_2 \frac{c- R_1 \ze}{R_2} + t_3 \cos(\vartheta) \sqrt{(1-\ze^2)\left(1- \left(\frac{c- R_1 \ze}{R_2}\right)^2\right)} +  t_4 \frac{c- R_1 \ze}{R_2} \ze \\
 & = \frac{t_2 c}{R_2} + \frac{t_1 R_2 - t_2 R_1 + t_4 c}{R_2} \ze - \frac{t_4 R_1}{R_2} \ze^2 + t_3 \cos(\vartheta) \sqrt{(1-\ze^2) \left(1 - \left(\frac{c- R_1\ze}{R_2}\right)^2\right)}.
\end{align*}
We abbreviate the term under the last root by
$$
A(\ze):= A(\ze,c,R_1,R_2) :=(1-\ze^2) \left(1 - \left(\frac{c- R_1\ze}{R_2}\right)^2\right)
$$
and its derivatives with respect to $\ze$ by 
\[
 A' := \partial_\zeta A\qquad \textrm{and}\qquad A'':=\partial_{\zeta\zeta}^2 A. 
\]
We note $A(\ze) \geq 0$ with $A(\ze)=0$ if and only if $\ze = \pm 1$ or $\ze = \frac{c\pm R_2}{R_1}$. 
Because of the bounds on $\ze$ we always have $A(\ze) > 0$. In order to find the critical points of $H$ on the symplectic quotient we calculate the partial derivatives
\begin{align*}
 \del_\vartheta H(\ze, \vartheta) & = -t_3 \sin(\vartheta) \sqrt{A(\ze)}, \\
 \del_\ze H(\ze, \vartheta) & =\frac{t_1 R_2 - t_2 R_1 + t_4 c}{R_2} - \frac{2 t_4 R_1}{R_2} \ze +  t_3 \cos(\vartheta) \frac{A'(\ze)}{2\sqrt{A(\ze)}}.
\end{align*}

\begin{lemma}
\label{critQuotient}
 $(\ze, \vartheta)$ is a critical point of $H$ on the symplectic quotient if and only if
 $$
 \vartheta \in \pi \Z \quad \mbox{and} \quad 0 = t_1 R_2 - t_2 R_1 + t_4 c - 2 t_4 R_1 \ze + t_3 R_2 \cos(\vartheta) \frac{A'(\ze)}{2\sqrt{A(\ze)}}.
 $$
\end{lemma}

\begin{proof}
 The point $(\ze, \vartheta)$ is critical if and only if $\del_\vartheta H(\ze, \vartheta) =0 = \del_\ze H(\ze, \vartheta)$. 
 Since $A(\ze)$ and $t_3$ are nonzero $ \del_\ze H(\ze, \vartheta) =0$ is equivalent to $\sin(\vartheta)=0$ meaning $\vartheta \in \pi \Z$ and $\cos(\vartheta) = \pm 1$. Together with $\del_\vartheta H(\ze, \vartheta) =0$ we get the desired result.
\end{proof}


\subsection{Integrability}

We consider the system \eqref{eqn_thesystem} in the chart $\phi=\phi_{e_1 e_2}$ defined in \eqref{chart}. By means of \eqref{HamVF}, we obtain as Hamiltonian vector fields in these coordinates
\begin{align*}
 X^{x_i} = -\frac{z_i}{R_i} \del_{ y_i}, \quad 
 X^{y_i} = \frac{z_i}{R_i} \del_{ x_i}, \quad
 X^{z_i} = -\frac{y_i}{R_i} \del_{ x_i}  + \frac{x_i}{R_i}\del_{ y_i}
\end{align*}
for $i \in \{1,2\}$. Moreover, we deduce from \eqref{dz}
\begin{equation*}
 dz_i = \frac{-x_i}{z_i} d x_i + \frac{-y_i}{z_i} d y_i.
\end{equation*}
This yields
\begin{equation}
\label{poissonCalc}
 \left\{ \quad
\begin{aligned}
 \{z_i, x_i\} & = - dz_i(X^{x_i}) = -\frac{y_i}{R_i}, \\  
 \{z_i, y_i\} & = - dz_i(X^{y_i}) =\frac{x_i}{R_i}, \\
 \{z_i, x_j\} & = \{z_i, y_j\} = 0 \quad \mbox{for } i \neq j, \\
  \{z_i, z_i\} & = - dz_i(X^{z_i}) = - \frac{x_i y_i}{R_i z_i} + \frac{x_i y_i}{R_i z_i} =0, \\
  \{z_1, z_2\} & = - dz_1(X^{z_2}) = 0.
\end{aligned}
\right. 
\end{equation}

Now we are ready to show

\begin{lemma}
\label{poissonNull}
 $\{J, H\}= 0$ for all $R_1, R_2 \in \R^{>0}$ and all $t=(t_1, t_2, t_3, t_4) \in \R^4$.
\end{lemma}

\begin{proof}
 We recall that the Poisson bracket is linear and that we have the identities in \eqref{poissonCalc}. Then we compute in the coordinates given in \eqref{chart}
 \begin{align*}
  \{J, H\} & = \{ R_1 z_1 + R_2 z_2 , \ t_1 z_1 + t_2 z_2 + t_3 (x_1 x_2 + y_1 y_2) + t_4 z_1 z_2 \} \\
  & \stackrel{\eqref{poissonCalc}}{=}  R_1 (t_3\{ z_1, x_1 x_2\} + t_3 \{z_1, y_1 y_2\} + t_4 \{z_1, z_1 z_2\} ) \\
   & \quad + R_2 ( t_3\{ z_2, x_1 x_2\} + t_3 \{z_2, y_1 y_2\} + t_4 \{z_2, z_1 z_2\} ).
 \end{align*}
Since the Poisson bracket satisfies the product rule $\{a, bc\} = \{a,b\} c + \{a,c\}b$ it follows
\begin{align*}
 & \stackrel{\eqref{poissonCalc}}{=} R_1 \left( t_3 \left(\frac{-y_1}{R_1}x_2 + \frac{x_1}{R_1} y_2 \right) \right)
 + R_2 \left( t_3 \left( \frac{-y_2}{R_2} x_1 +  \frac{x_2}{R_2} y_1 \right) \right) 
  = 0.
\end{align*}
The charts in \eqref{chart} are not defined for $z_i=0$. To show $\{J, H\}=0$ there, consider Cartesian coordinates $(x_1, y_1, z_1, x_2, y_2, z_2)$ and choose charts given by 
$$(x_1, z_1, x_2, z_2) \mapsto \left(x_1, e_1\sqrt{1-x_1^2 - z_1^2}, z_1, x_2,  e_2\sqrt{1-x_2^2 - z_2^2}, z_2 \ \right)$$ 
etc. The calculations are completely analogous.
\end{proof}

\begin{proposition}
\label{prop_integrable}
 The system $(\mbS^2\times\mbS^2,\om,(J_{\vec{R}},H_{\vec{t}}))$ given in Equation~\eqref{eqn_thesystem} is integrable for all parameter values
 with $0<R_1<R_2$ and $t_1,t_2,t_3,t_4\in\R$ with $t_3\neq 0$.
\end{proposition}

\begin{proof}
 Fix parameter values and let $J = J_{(R_1,R_2)}$ and $H = H_{(t_1, t_2, t_3, t_4)}$
 with $t_3\neq 0$.
 In Lemma~\ref{poissonNull} we showed that $\{J,H\}=0$.
 It follows from Lemma~\ref{critQuotient} and the fact that $t_3\neq 0$
 that the rank 1 critical
 points occupy a set of measure zero since there are only finitely many
 on each symplectic quotient.
 By Lemma~\ref{rankNull} there are only finitely many rank 0 points and
 thus $J$ and $H$ are linearly independent almost everywhere.
\end{proof}

\subsection{Nondegeneracy of rank 1 points}

Now we want to study nondegeneracy of the 
rank $1$ critical points. Therefore we have to compute the Hessian of $H$ on the symplectic quotient. We get
\begin{align*}
 \del^2_{\vartheta \vartheta}H(\ze, \vartheta) & =-t_3 \cos(\vartheta) \sqrt{A(\ze)},\\
 \del^2_{\vartheta \ze}H(\ze, \vartheta) & = -t_3 \sin(\vartheta)\frac{A'(\ze)}{2\sqrt{A(\ze)}} ,\\
  \del^2_{\ze \ze} H(\ze, \vartheta) & =- \frac{2 t_4 R_1}{R_2} + \frac{t_3 \cos(\vartheta)}{2} \ \frac{A''(\ze) \sqrt{A(\ze)} - A'(\ze) \frac{A'(\ze)}{2 \sqrt{A(\ze)}}}{A(\ze)} \\
  & =\frac{- 2 t_4 R_1}{R_2} + t_3 \cos(\vartheta) \ \frac{2A''(\ze)A(\ze) - (A'(\ze))^2 }{4 (A(\ze))^{\frac{3}{2}}}
\end{align*}
Now we come to a criterion for nondegeneracy.
%
%
%
Let $pr_c\colon J^{-1}(c)\to J^{-1}(c)/\mathbb{S}^1$ denote the quotient map
for each $c\in J(\mathbb{S}^2\times\mathbb{S}^2)$.
\begin{proposition}[{\bf Rank 1 Criterion}]\label{prop_criterion}
Suppose $p\in \mathbb{S}^2\times\mathbb{S}^2$ is a rank 1 critical point and denote $c = J(p)$ and $pr_c(p) = (\ze,\vartheta)$.
Then $p$ is non-degenerate if and only if $\del^2_{\ze \ze} H(\ze, \vartheta) \neq 0$.
In particular, $p$ is non-degenerate and of elliptic-regular type if
 \[
  \frac{2t_4R_1}{t_3R_2}\cos(\vartheta) > \frac{2A''(\ze)A(\ze) - (A'(\ze))^2 }{4 (A(\ze))^{\frac{3}{2}}},
 \]
 non-degenerate and of hyperbolic-regular type if
 \[
  \frac{2t_4R_1}{t_3R_2}\cos(\vartheta) < \frac{2A''(\ze)A(\ze) - (A'(\ze))^2 }{4 (A(\ze))^{\frac{3}{2}}},
 \]
 and degenerate otherwise.
\end{proposition}

\begin{proof}
We start by computing the symplectic form on the symplectic quotient.
 Let $j\colon J^{-1}(c)\to \mathbb{S}^2\times\mathbb{S}^2$ be the inclusion map.
 Recall $\om = \sum_{i=1}^2 R_i dz_i\wedge d\theta_i$ so we have
\[
  j^*\om = R_1 dz_1\wedge d\theta_1 + R_2d\left(\frac{c-R_1 z_1}{R_2}\right)\wedge d\theta_2
  = R_1 dz_1 \wedge d(\theta_1-\theta_2)
\]
 and thus on the reduced space $\mathbb{S}^2\times\mathbb{S}^2\sslash \mathbb{S}^1$ in the coordinates $(\zeta,\vartheta)$ we have the symplectic form
 \[
  \om_{\mathrm{red}} = R_1 d\zeta\wedge d\vartheta
\qquad
 \textrm{with matrix}
\qquad
  \om_{\mathrm{red}} = R_1 \begin{pmatrix} 0 & 1 \\ -1 & 0\end{pmatrix}.
 \]
 Since $(\zeta, \vartheta)$ is a critical point \refcritQuotient\ implies that $\sin(\vartheta)=0$,
 so $\partial_{\vartheta \zeta} H(\zeta,\vartheta) = 0$ and thus
 \begin{align*}
  \om_{\mathrm{red}}^{-1}d^2 H(\zeta, \vartheta) &= \frac{1}{R_1}\begin{pmatrix}0&-1\\1&0\end{pmatrix}\begin{pmatrix}\partial^2_{\zeta \zeta} H(\zeta,\vartheta) & 0\\ 0 & \partial^2_{\vartheta \vartheta} H(\zeta,\vartheta)\end{pmatrix}\\
                                                 &=\frac{1}{R_1}\begin{pmatrix} 0 & -\partial_{\theta \theta}^2 H(\zeta,\vartheta)\\ \partial_{\zeta \zeta}^2 H(\zeta,\vartheta) & 0 \end{pmatrix}
 \end{align*}
 which has eigenvalues
 \[
  \lambda_{\pm} = \pm\frac{1}{R_1}\sqrt{-\partial^2_{\zeta \zeta} H(\zeta,\vartheta) \partial^2_{\vartheta \vartheta} H(\zeta, \vartheta)}.
 \]
 Since $\cos(\vartheta) = \pm 1$ we see that $\partial^2_{\vartheta \vartheta} H(\ze,\vartheta)\neq 0$ and so the eigenvalues
 are distinct if and only if $\partial^2_{\ze \ze} H(\ze,\vartheta)\neq 0$, establishing the first part of the claim.
 To complete the proof we notice that $\lambda_\pm$ are purely imaginary if $\partial^2_{\zeta \zeta} H(\zeta,\vartheta) \partial^2_{\vartheta \vartheta} H(\zeta, \vartheta)>0$, which implies that $p$ is elliptic-regular, and purely real otherwise,
 implying that $p$ is hyperbolic-regular.
 We compute
 \begin{align*}
  &\partial^2_{\zeta \zeta} H(\zeta,\vartheta) \partial^2_{\vartheta \vartheta} H(\zeta, \vartheta) \\
     &=\left(\frac{- 2 t_4 R_1}{R_2} + t_3 \cos(\vartheta) \ \frac{2A''(\ze)A(\ze) - (A'(\ze))^2 }{4 (A(\ze))^{\frac{3}{2}}}\right)\left( -t_3 \cos(\vartheta) \sqrt{A(\ze)}\right)\\
     &=\sqrt{A(\ze)} \left(\frac{2 t_3 t_4 R_1}{R_2} \cos(\vartheta) - t_3^2  \ \frac{2A''(\ze)A(\ze) - (A'(\ze))^2 }{4 (A(\ze))^{\frac{3}{2}}}\right),
 \end{align*}
 and the result follows because $\sqrt{A(\ze)}>0$ for the bounds on $\ze$.
\end{proof}

The following is established by plugging the specific values
into the inequality from Proposition~\ref{prop_criterion}
(for more details see the proof of Lemma~\ref{lem_rank1nondegen}).

\begin{corollary}\label{cor_rank1nondegents}
 For the parameter values
 \begin{equation}
  R_1 = 1, \quad R_2 = 2, \quad t_1 = \frac{1}{4}, \quad t_2 = \frac{1}{4}, \quad t_3 = \frac{1}{2}, \quad t_4 = 0,
 \end{equation}
 all rank 1 points are non-degenerate and of elliptic-regular type.
\end{corollary}

%


\section{A linear combination of systems of toric type}
\label{sec_s1s2}


In this section we apply the results of the previous section to a special choice of parameters
of the system.
Let $\vec{s} := (s_1,s_2)\in[0,1]^2$ and consider the system $(J_{\vec{R}}, H_{\vec{s}})$ on $\mbS^2\times\mbS^2$ using the same $J_{\vec{R}}$ as before
but using $H_{\vec{s}} :=H_{\vec{t}}$ where 
\[
t_1 = (1-s_1)(1-s_2),\qquad t_2 = s_1 s_2,\qquad t_3= s_1+s_2-2 s_1 s_2,\qquad t_4 = s_1-s_2,
\]
i.e., we consider
\begin{equation}\label{eqn_s1s2system}
 \left\{
  \begin{aligned}
   J_{\vec{R}}(x_1, y_1, z_1 , x_2, y_2, z_2) := \,& R_1 z_1 + R_2 z_2 , \\
   H_{\vec{s}}(x_1, y_1, z_1 , x_2, y_2, z_2)  := \, & (1-s_1)(1-s_2) z_1 + s_1 s_2 z_2 \\&+s_1(1-s_2)(x_1 x_2 + y_1 y_2+z_1 z_2) \\& + s_2(1-s_1)(x_1 x_2 + y_1 y_2 - z_1 z_2).
  \end{aligned}
 \right.
\end{equation}

Thinking of $R_1$ and $R_2$ as fixed, this produces a 
two parameter family of systems $\{(J_{\vec{R}}, H_{(s_1,s_2)})\mid s_1,s_2\in[0,1]\}$.
This family is of interest because it shows the system $(J_{\vec{R}}, H_{\left(\frac{1}{2},\frac{1}{2}\right)})$,
which is a semitoric integrable system with exactly two focus-focus
points by Theorem~\ref{mainThm}, as a linear
combination of systems of toric type.
The systems $(J_{\vec{R}},H_{(0,0)})$, $(J_{\vec{R}},H_{(0,1)})$, $(J_{\vec{R}},H_{(1,0)})$, 
and $(J_{\vec{R}},H_{(1,1)})$ are systems of toric type whose associated polygons
agree (as subsets of $\R^2$) with four elements of the semitoric polygon of the 
semitoric system $(J_{\vec{R}},H_{\left(\frac{1}{2},\frac{1}{2}\right)})$.
The images of the momentum maps for these systems are shown in Figure~\ref{fig_array}
and a plot describing the number of focus-focus points for different values of
$s_1,s_2\in [0,1]$ is show in Figure~\ref{fig_FFs1s2}.

In the following series of lemmas we apply the various general results 	developed in Section~\ref{sec_familyofsystems}
to the special case of the system~\eqref{eqn_s1s2system}.

\begin{lemma}\label{lem_s1s2integrable}
 For any choice of parameters $s_1, s_2\in [0,1]$ the system in Equation~\eqref{eqn_s1s2system}
 is integrable.
\end{lemma}

\begin{proof}
 Recall that~\eqref{eqn_s1s2system} is a special case of~\eqref{eqn_thesystem} with $t_3 = s_1+s_2-2s_1s_2$,
 so by Proposition~\ref{prop_integrable} we know the result holds for all $s_1,s_2\in [0,1]$ such that
 $s_1+s_2-2s_1s_2\neq 0$.  This only leaves the cases of $s_1 = s_2 = 0$ and $s_1 = s_2 = 1$.
 The case $s_1 = s_2 = 0$ leads to the system
 \[
  J_{(1,2)}(x_1, y_1, z_1, x_2, y_2, z_2) = z_1 + 2 z_2,\qquad H_{(0,0)} (x_1, y_1, z_1, x_2, y_2, z_2)=z_1
 \]
 and the case $s_1 = s_2 = 1$ to the system 
 \[
  J_{(1,2)}(x_1, y_1, z_1, x_2, y_2, z_2) = z_1 + 2 z_2,\qquad H_{(1,1)} (x_1, y_1, z_1, x_2, y_2, z_2)=z_2,
 \]
 which are each known to be toric integrable systems.
\end{proof}

\begin{lemma}
\label{lem_rank1nondegen}
 For any choice of parameters $s_1, s_2\in [0,1]$, all rank 1 critical points of $(J_{(1,2)},H_{(s_1,s_2)})$
  are nondegenerate and of elliptic-regular type.
\end{lemma}

\begin{proof}
 The cases of $s_1 = s_2 = 0$ and $s_1 = s_2 = 1$ produce toric systems as described in the proof of Lemma~\ref{lem_s1s2integrable},
 so all rank 1 points in these systems are non-degenerate and of elliptic-regular type.
 Now consider $(s_1,s_2)\in [0,1]^2\setminus \{(0,0), (1,1)\}$ which implies $s_1+s_2-2s_1s_2>0$.
 Substituting $R_1 = 1$, $R_2 = 2$, $t_3 = s_1+s_2-2s_1s_2$ and $t_4 = s_1-s_2$ 
 into the criterion in Proposition~\ref{prop_criterion} we see that it is sufficient to show
 \begin{equation}\label{eqn_s1s2rank1}
  \frac{s_1-s_2}{s_1+s_2-2s_1s_2}\cos(\vartheta) > \frac{2A''(\ze)A(\ze) - (A'(\ze))^2 }{4 (A(\ze))^{\frac{3}{2}}}.
 \end{equation}
 Standard calculus shows that the value of the left-hand-side of Equation~\eqref{eqn_s1s2rank1} is
 in the interval $[-1,1]$ for all $s_1,s_2,\vartheta$ and the value of the right-hand-side of
 Equation~\eqref{eqn_s1s2rank1} can be seen to be in the interval $]-\infty,-1[	$ 
 for all $(\ze,c)\in \ ]-1,1[ \ \times\ ]-3,3[$ by plotting it in Mathematica
 (see Figure~\ref{newBplots}), so the inequality is verified.
\end{proof}

\begin{figure}[h]
\centering
    \begin{subfigure}[b]{0.4\textwidth}
        \includegraphics[width=\textwidth]{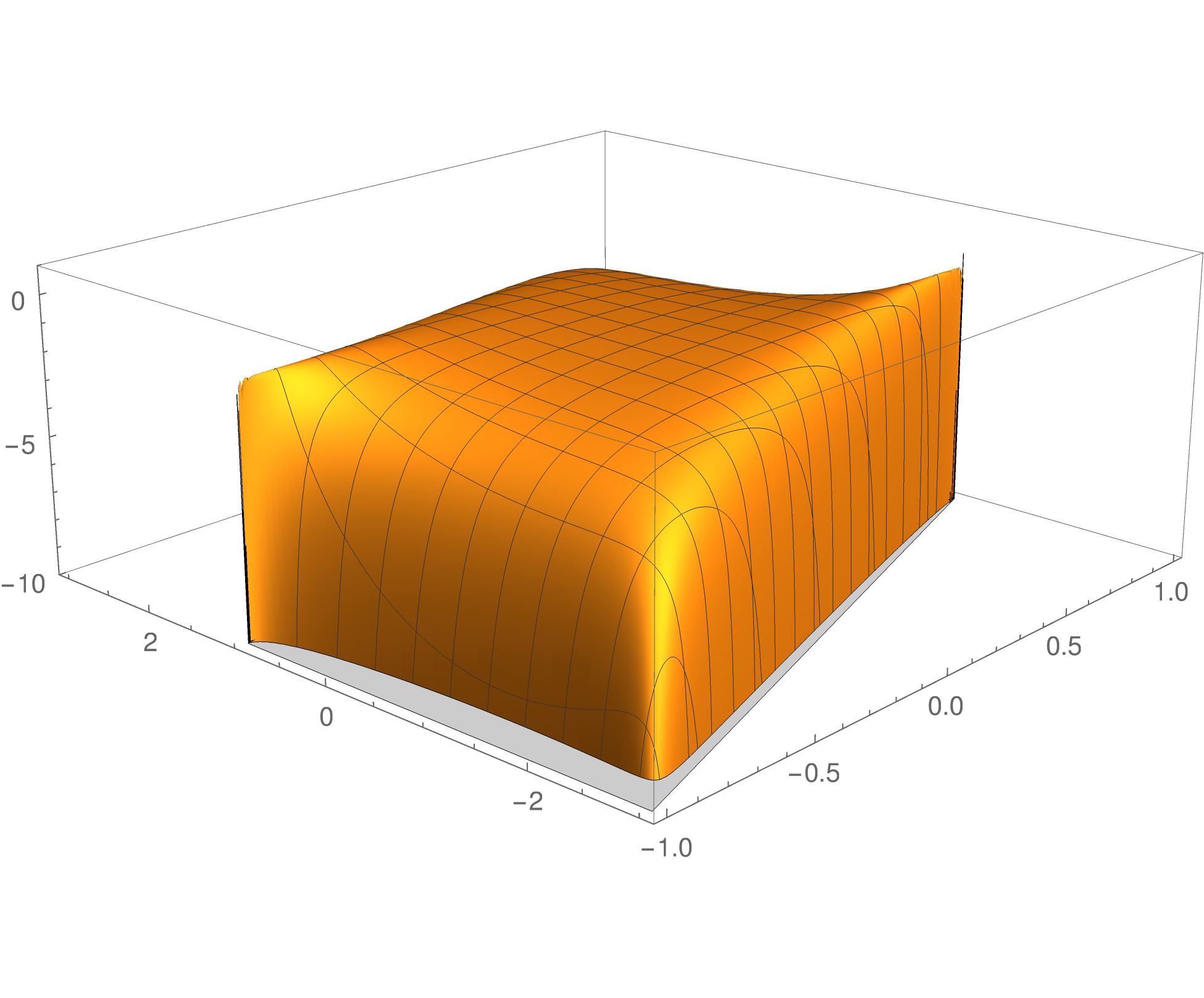}
    \end{subfigure}
    \qquad \quad
    \begin{subfigure}[b]{0.3\textwidth}
        \includegraphics[width=\textwidth]{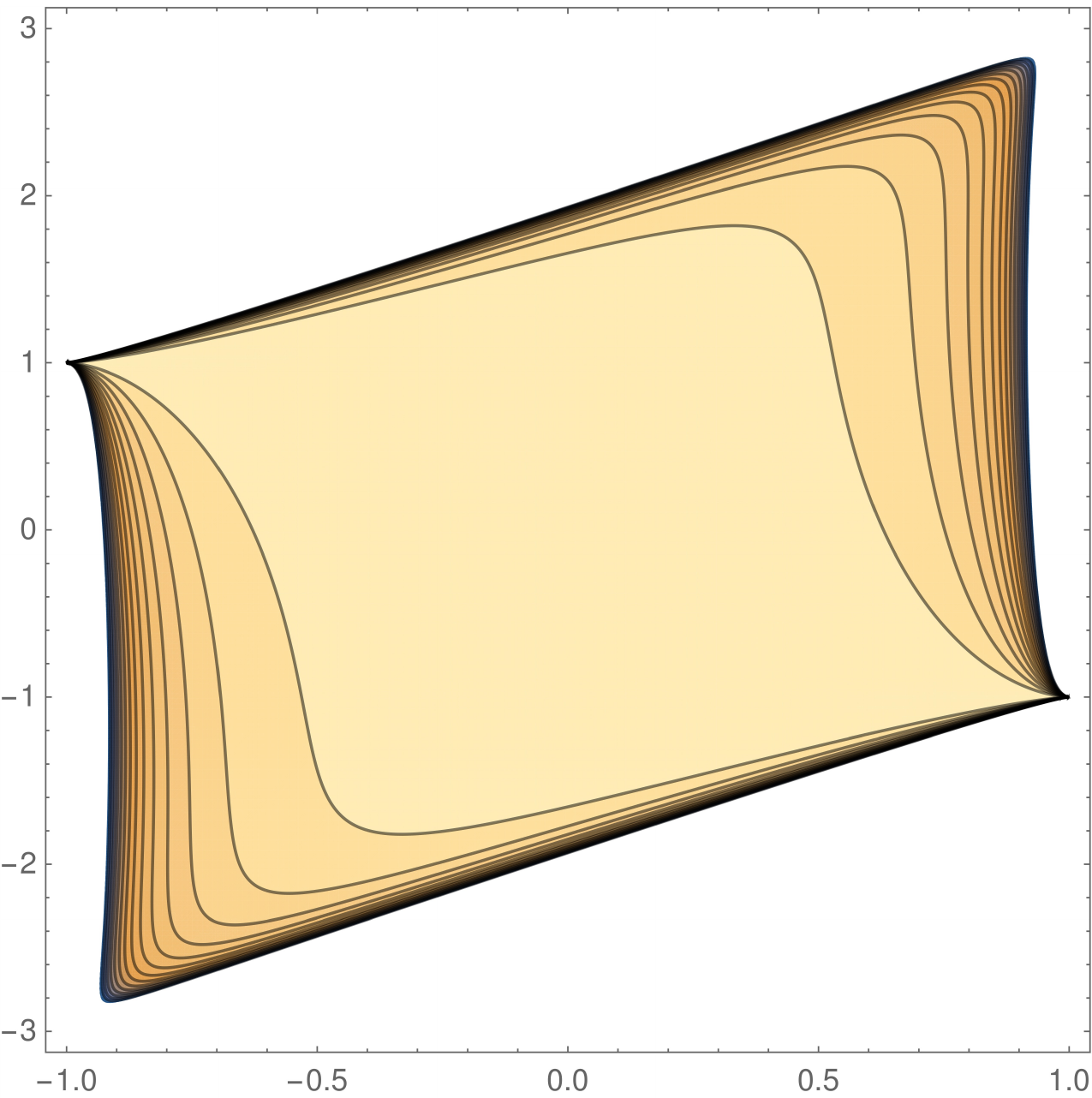}
    \end{subfigure}
     \quad
    \begin{subfigure}[b]{0.05\textwidth}
        \includegraphics[width=\textwidth]{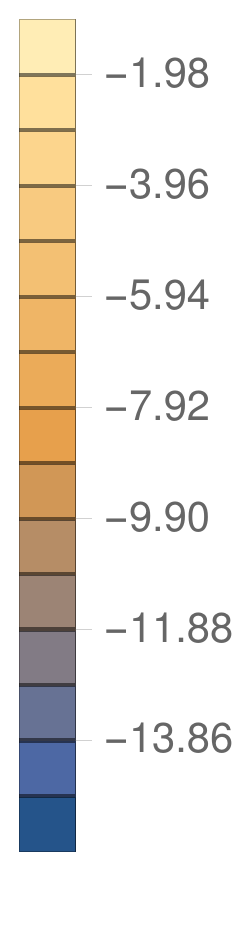}
    \end{subfigure}
 \caption{This figure analyses the right hand side of Equation~\eqref{eqn_s1s2rank1}: The plot on the left shows the graph of the right hand side of Equation~\eqref{eqn_s1s2rank1} which is always below $-1.06066$. The contour plot on the right displays the associated level sets.}
 \label{newBplots}   
\end{figure}

For $R_1=1$, $R_2=2$, $\vec{s}=(s_1,s_2)\in[0,1]^2$ and $\vec{e}=(e_1,e_2)\in\{0,1\}^2$ consider the discriminant from \eqref{eqn_discrimiant} given by
\begin{align*}
 \lapla_{(\vec{s},\vec{e})} & := \bigg(\frac{1}{ 4} \Big( ( s_1 s_2 + e_1 (s_1-s_2))^2 + 4 e_1 e_2 (s_1+s_2-2 s_1 s_2)^2 \\
 & \qquad + 4((1-s_1)(1-s_2) + e_2 (s_1-s_2))^2 \Big)\bigg)^2  \\
 & \quad - \bigg( ( s_1 s_2 + e_1 (s_1-s_2))^ 2 ((1-s_1)(1-s_2) + e_2 (s_1-s_2))^2  \\
 & \qquad - 2 e_1 e_2( s_1 s_2 + e_1 (s_1-s_2))((1-s_1)(1-s_2) + e_2 s_1-s_2) (s_1+s_2-2 s_1 s_2)^2 \\
 & \qquad + (s_1+s_2-2 s_1 s_2)^4  \bigg)	
\end{align*}
and set 
\begin{align}
 \gamma_{(N,S)} & := \{(s_1,s_2)\in [0,1]^2\mid \lapla_{(s_1,s_2,1,-1)}=0\}, \nonumber\\
 \gamma_{(S,N)} & := \{(s_1,s_2)\in [0,1]^2\mid \lapla_{(s_1,s_2,-1,1)}=0\}, \label{def_gamma}\\
 \gamma & : = \gamma_{(N,S)}\cup\gamma_{(S,N)}.\nonumber
\end{align}
The sets are plotted in Figure~\ref{fig_FFs1s2}.

\begin{lemma}\label{lem_2FF-s1s2}
The system $(J_{(1,2)},H_{(s_1,s_2)})$, $s_1,s_2\in[0,1]$, has exactly four critical points of rank $0$, namely
$\{(N,N),(N,S),(S,N),(S,S)\}$. The points $(N,N)$ and $(S,S)$ are non-degenerate and of
elliptic-elliptic type for all $s_1,s_2\in [0,1]^2$.
The point $(N,S)$ is non-degenerate except when $(s_1,s_2)\in\gamma_{(N,S)}$
and the point $(S,N)$ is non-degenerate except when $(s_1,s_2)\in\gamma_{(S,N)}$.
In particular, for $s_1, s_2 \in \{0,1\}$, all four points are elliptic-elliptic and
for $s_1=s_2=\frac{1}{2}$ the points $(N,S)$ and $(S,N)$ are both focus-focus.
\end{lemma}

\begin{proof}
Using Corollary \ref{discrim}, we study the behaviour of the discriminant $\lapla_{(\vec{s},\vec{e})}$ for the parameter values in question.
If $(e_1, e_2) \in \{(1,1), (-1,-1)\}$, we are in the chart 
around $(N,N)$ or $(S, S)$ and $\lapla_{(\vec{s},\vec{e})}$ 
is positive. Figure \ref{nondegFixedPoints} shows a plot 
of the case $(e_1, e_2)=(1,1)$.
If $(e_1, e_2) \in \{(1,-1), (-1,1)\}$, we are in the chart 
around $(N,S)$ or $(S,N)$ and $\lapla_{(\vec{s},\vec{e})}$ 
vanishes along two curves. Figure \ref{nondegFixedPoints2} shows a
plot of the case $(e_1, e_2)=(1,-1)$.
\begin{figure}[h]
\centering
    \begin{subfigure}[b]{0.4\textwidth}
        \includegraphics[width=\textwidth]{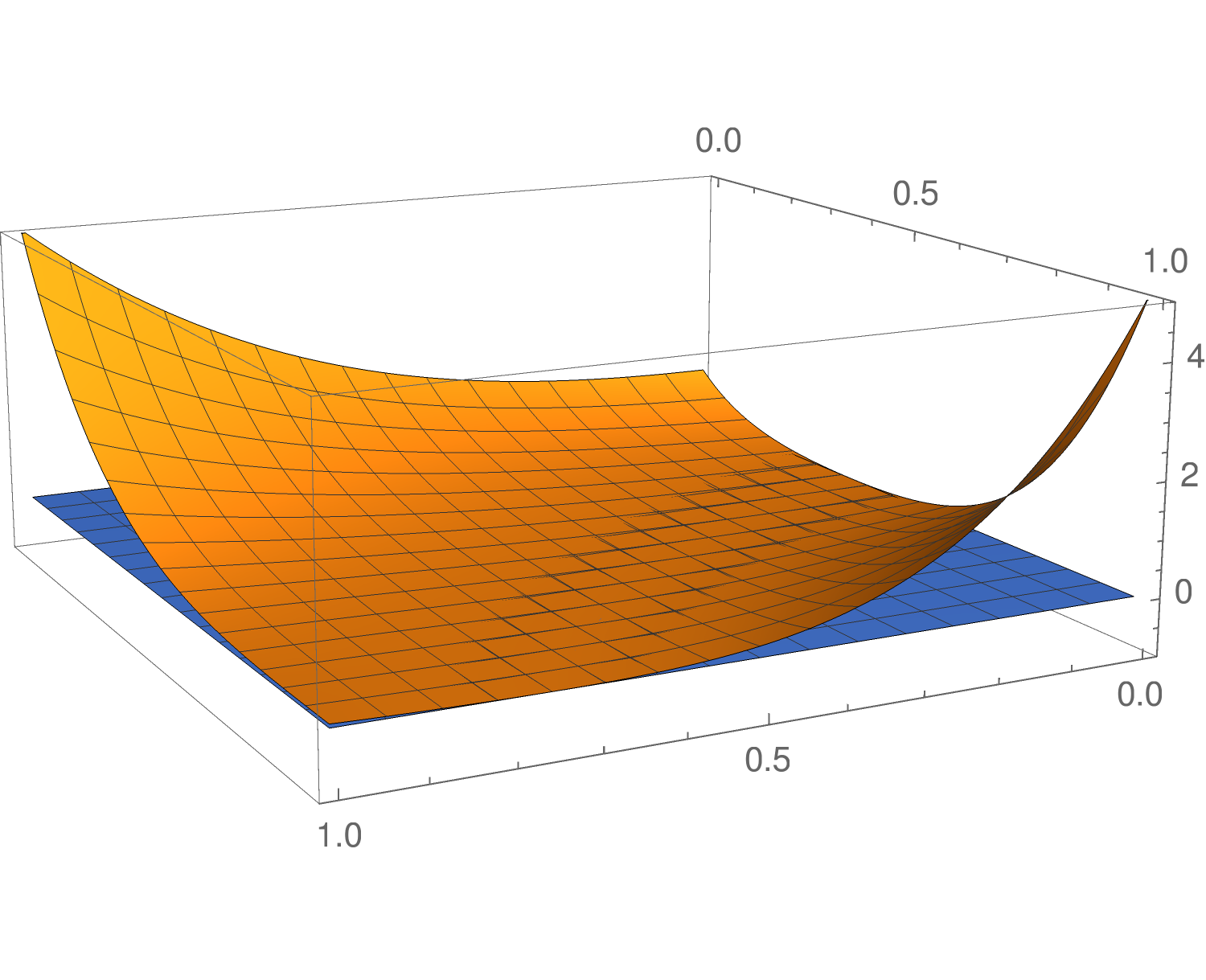}
    \end{subfigure}
    \qquad \quad
    \begin{subfigure}[b]{0.3\textwidth}
        \includegraphics[width=\textwidth]{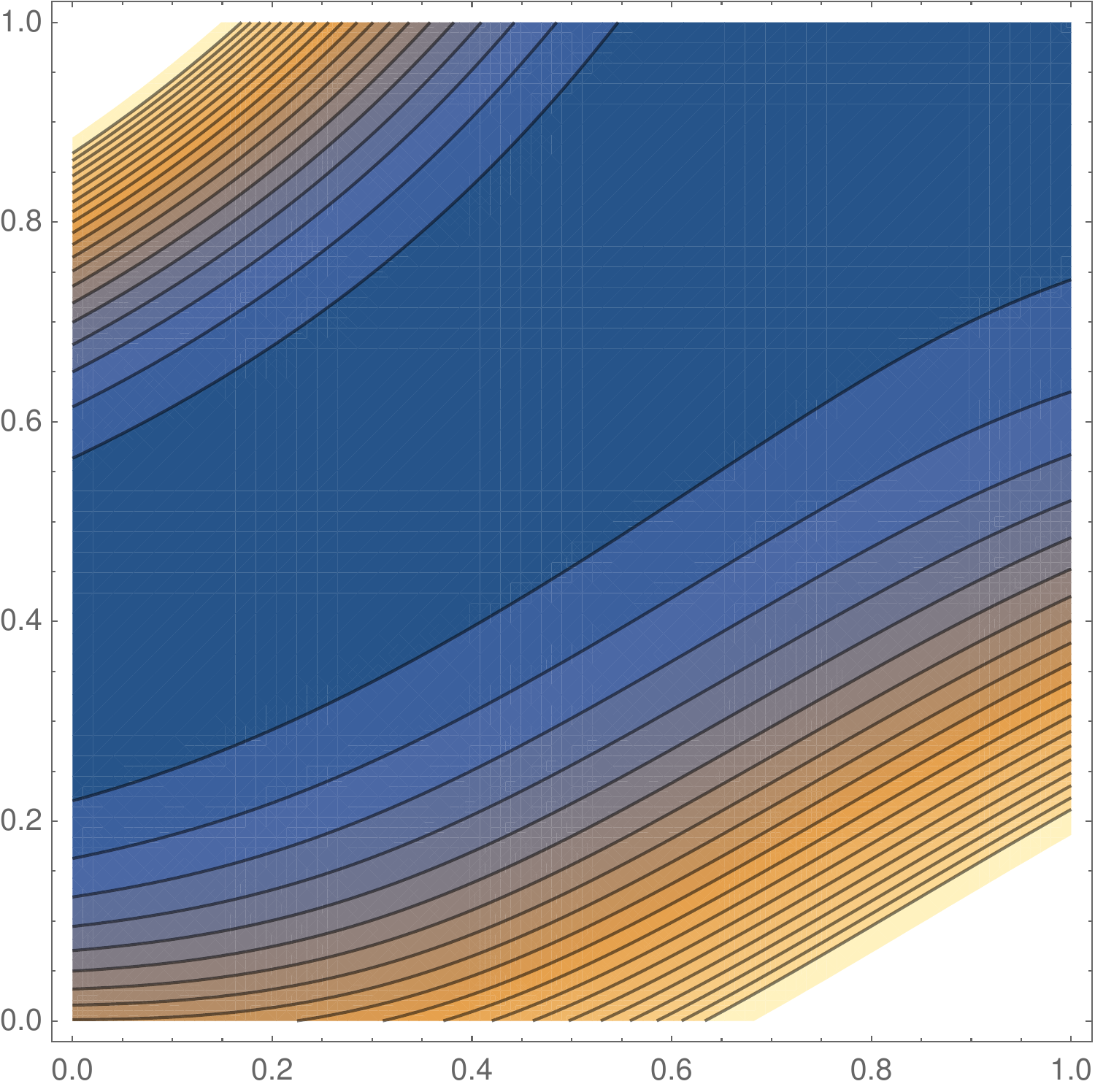}
    \end{subfigure}
     \quad
    \begin{subfigure}[b]{0.05\textwidth}
        \includegraphics[width=\textwidth]{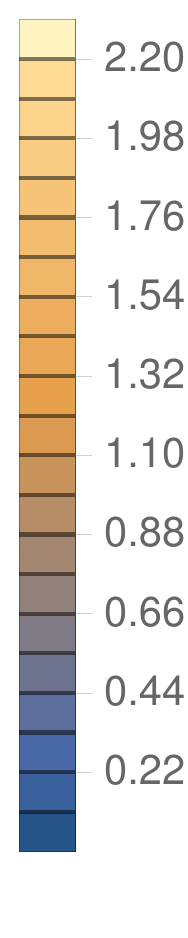}
    \end{subfigure}
 \caption{Case $(e_1, e_2)=(1,1)$: on the left, the graph of $ (s_1, s_2) \mapsto \lapla_{((s_1, s_2),(1,1))}$ (orange) and a plane through zero (blue) are displayed. On the right, the associated level sets of $ (s_1, s_2) \mapsto \lapla_{((s_1, s_2),(1,1))}$ are shown.}
 \label{nondegFixedPoints}   
\end{figure}
\begin{figure}[h]
\centering
    \begin{subfigure}[b]{0.4\textwidth}
        \includegraphics[width=\textwidth]{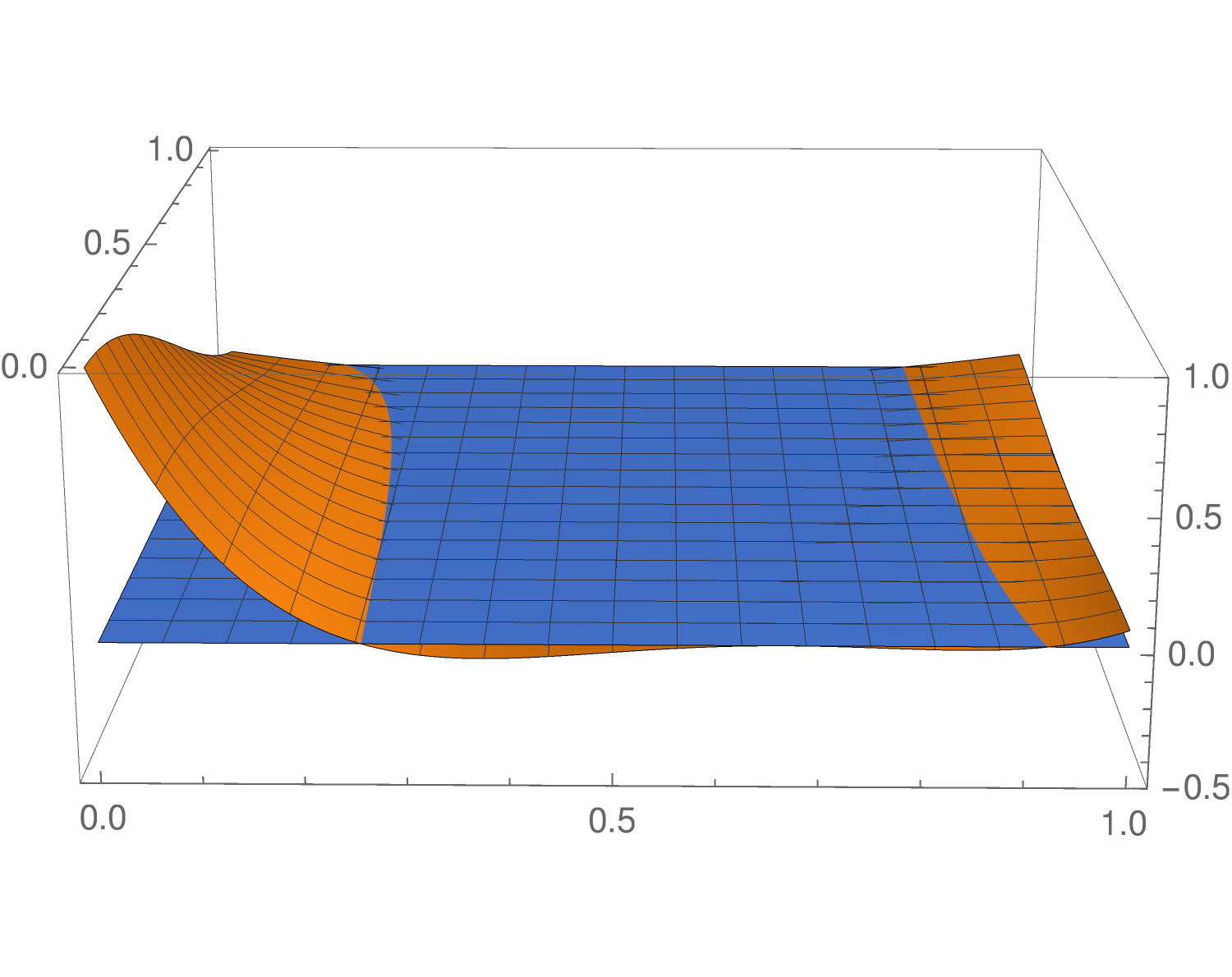}
    \end{subfigure}
    \qquad \quad
    \begin{subfigure}[b]{0.3\textwidth}
        \includegraphics[width=\textwidth]{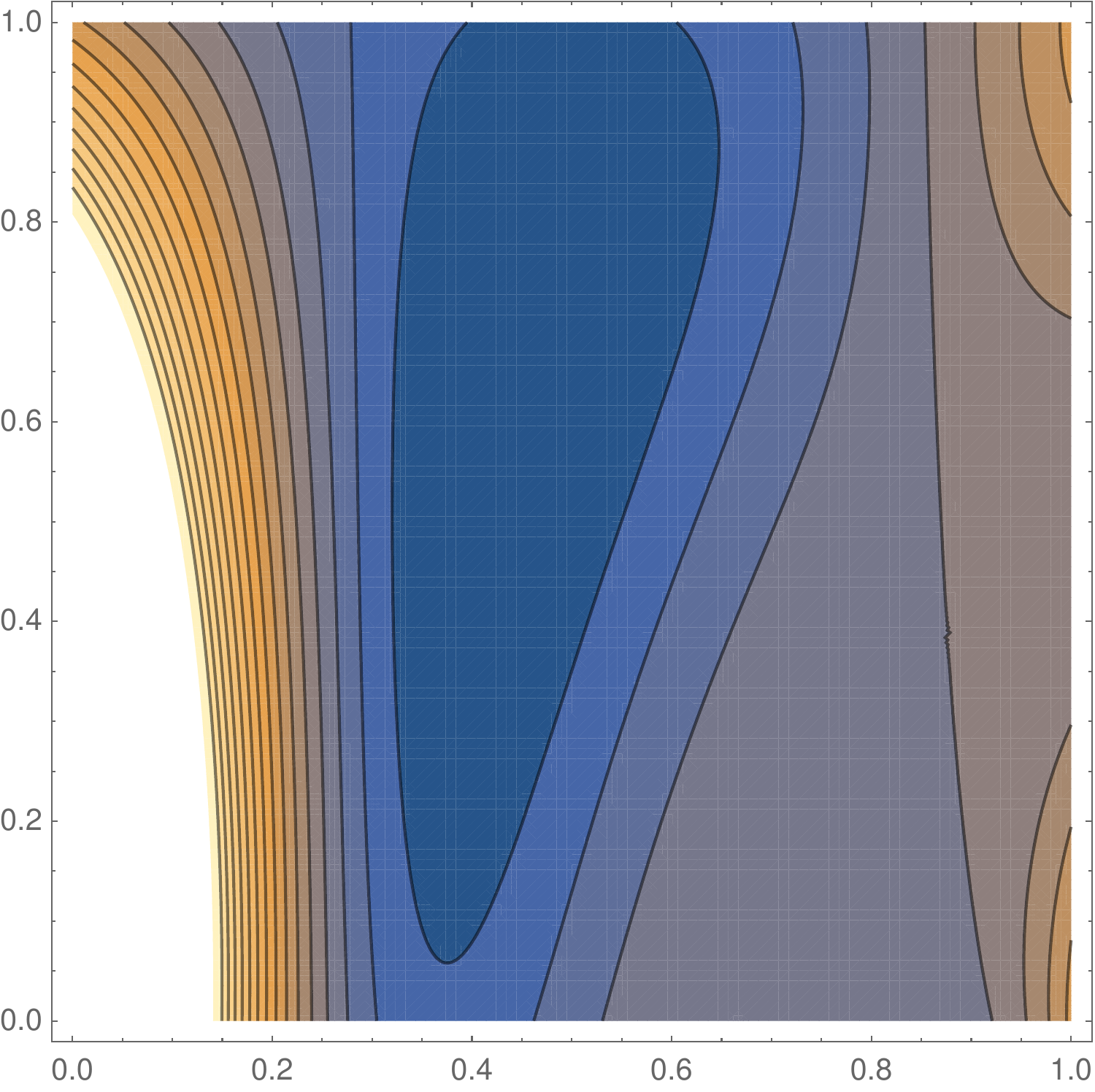}
    \end{subfigure}
     \quad
    \begin{subfigure}[b]{0.05\textwidth}
        \includegraphics[width=\textwidth]{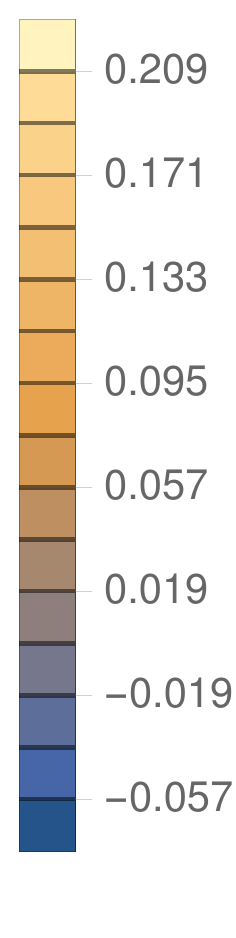}
    \end{subfigure}
 \caption{Case $(e_1, e_2)=(1,-1)$: on the left, the graph of $ (s_1, s_2) \mapsto \lapla_{((s_1, s_2),(1,-1))}$ (orange) and a plane through zero (blue) are displayed. On the right, the associated level sets of $ (s_1, s_2) \mapsto \lapla_{((s_1, s_2),(1,-1))}$ are shown.}
  \label{nondegFixedPoints2}   
\end{figure}
\end{proof}

Lemmas~\ref{lem_s1s2integrable},~\ref{lem_rank1nondegen}, and~\ref{lem_2FF-s1s2} combine to form the following, which implies Theorem~\ref{thm_s1s2intro}.

\begin{theorem}\label{thm_s1s2full}
The system $(J_{(1,2)},H_{(s_1,s_2)})$ has the following properties: 
\begin{enumerate}
 \item for all $s_1,s_2\in[0,1]^2$ it is an integrable system 
 such that, with the possible exception of $(N,S)$ and $(S,N)$ (depending on $s_1$ and $s_2$),
 all of the singular points are non-degenerate of type elliptic-elliptic or elliptic-regular\emph{;}
 \item the points $(N,S)$ and $(S,N)$ are rank 0 singular points which transition between
being of focus-focus, elliptic-elliptic, and degenerate as $(s_1,s_2)$ varies, and they
are only degenerate on a set $\gamma\subset [0,1]^2$ which is the union of four smooth curves.
\end{enumerate}
Thus, $(J_{(1,2)},H_{(s_1,s_2)})$ is a semitoric system
for all $(s_1,s_2)\in [0,1]^2\setminus\gamma$.
In particular, if $(s_1,s_2)\in \{(0,0),(0,1),(1,0),(1,1)\}$ then $(J_{(1,2)},H_{(s_1,s_2)})$
is a semitoric system with no focus-focus points and the system $(J_{(1,2)},H_{(\frac{1}{2},\frac{1}{2})})$
is a semitoric system with exactly two focus-focus points.
\end{theorem}

Note that the set $\gamma$ is given in Equation~\eqref{def_gamma} and is plotted in
Figure~\ref{fig_FFs1s2}.

\subsection{A degenerate point}
By Proposition~\ref{prop_degenerate} we know that for each $\vec{R}$ there exist some values of $s_1,s_2\in [0,1]$
such that $(J_{\vec{R}},H_{(s_1,s_2)})$ is a degenerate system
because the points $(N,S)$ and $(S,N)$ transition between being focus-focus and being elliptic-elliptic.

\begin{example}
Assume that $s_1=s_2$.  Since $(J_{(1,2)},H_{(0,0)})$ and $(J_{(1,2)},H_{(1,1)})$ have 
no focus-focus points and $(J_{(1,2)}, H_{\left(\frac{1}{2},\frac{1}{2}\right)})$
has focus-focus points at $(N,S)$ and $(S,N)$ there must exist
at least two values of $s\in\ ]0,1[$ such that $(J_{(1,2)},H_{(s,s)})$ has a degenerate rank 0 point by Proposition~\ref{prop_degenerate}.
Plugging $t_1 = (1-s)^2$, $t_2 = s^2$, $t_3 = 2s(1-s)$, $t_4=0$, $e_1=-1$, and $e_2=1$ into $\om^{-1}_p d^2H$ in Equation~\ref{omHessH}
and taking the discriminant of the characteristic polynomial equal to zero gives exactly two solutions
in the range $]0,1[$. These solutions are $s_+$ and $s_-$ where
\[s_{\pm} = \frac{1}{31} \left(\pm 8 \sqrt{5}+14 \mp \sqrt{82 \mp 24 \sqrt{5}}\right)\]
and $s_+ \approx 0.856953$, $s_- \approx 0.250291$.
Since there must be at least two degenerate points and these are the only points for which
$\om^{-1}_p d^2H$ has less than four distinct eigenvalues we conclude that
$(J_{(1,2)}, H_{(s_+, s_+)})$ and $(J_{(1,2)},H_{(s_-,s_-)})$ have a degenerate
point at $(S,N)$.
\end{example}

\section*{Acknowledgments}
The first author was partially supported by the Research Fund of the University of Antwerp and the second
author is partially supported by an AMS-Simons travel grant. The second author is extremely grateful
to the IHES for inviting him to visit in the summer of 2017, where some of the work for this project
was completed. Additionally, both authors are thankful to
Yohann Le Floch and Jaume Alonso Fern\'andez for helpful discussions.
Moreover, we thank Taras Skrypnyk for bringing the Gaudin model to our attention and
for describing to us its relevance in mathematical physics.

\bibliographystyle{aims}
\bibliography{2focusfocusreferences}

\providecommand{\href}[2]{#2}
\providecommand{\arxiv}[1]{\href{http://arxiv.org/abs/#1}{arXiv:#1}}
\providecommand{\url}[1]{\texttt{#1}}
\providecommand{\urlprefix}{URL }
\begin{thebibliography}{10}

\bibitem{AFDH-spin}
\newblock J.~Alonso, H.~Dullin and S.~Hohloch,
\newblock Taylor series and twisting-index invariants of coupled
  spin-oscillators,
\newblock ArXiv:1712.06402.

\bibitem{Atiyah1982}
\newblock M.~F. Atiyah,
\newblock Convexity and commuting {H}amiltonians,
\newblock \emph{Bull. London Math. Soc.}, \textbf{14} (1982), 1--15.

\bibitem{bolsinov-fomenko}
\newblock A.~V. Bolsinov and A.~T. Fomenko,
\newblock \emph{Integrable {H}amiltonian systems},
\newblock Chapman \& Hall/CRC, Boca Raton, FL, 2004,
\newblock Geometry, topology, classification, Translated from the 1999 Russian
  original.

\bibitem{Montaldi_notes}
\newblock P.-L. Buono, F.~Laurent-Polz and J.~Montaldi,
\newblock Symmetric {H}amiltonian bifurcations,
\newblock in \emph{Geometric mechanics and symmetry}, vol. 306 of London Math.
  Soc. Lecture Note Ser.,
\newblock Cambridge Univ. Press, Cambridge, 2005,
\newblock 357--402,
\newblock Based on lectures by Montaldi.

\bibitem{CushmanBates-book}
\newblock R.~Cushman and L.~Bates,
\newblock \emph{Global aspects of classical integrable systems},
\newblock 2nd edition,
\newblock Birkh\"auser/Springer, Basel, 2015,
\newblock \urlprefix\url{https://doi.org/10.1007/978-3-0348-0918-4}.

\bibitem{CuVN2002}
\newblock R.~Cushman and V.~u~Ngoc~S.,
\newblock Sign of the monodromy for {L}iouville integrable systems,
\newblock \emph{Ann. Henri Poincar\'e}, \textbf{3} (2002), 883--894.

\bibitem{Delzant1988}
\newblock T.~Delzant,
\newblock Hamiltoniens p\'eriodiques et images convexes de l'application
  moment,
\newblock \emph{Bull. Soc. Math. France}, \textbf{116} (1988), 315--339.

\bibitem{dullin-pelayo}
\newblock H.~Dullin and A.~Pelayo,
\newblock Generating hyperbolic singularities in semitoric systems via {H}opf
  bifurcations,
\newblock \emph{J. Nonlinear Sci.}, \textbf{26} (2016), 787--811.

\bibitem{efstathiou}
\newblock K.~Efstathiou and N.~Martynchuk,
\newblock Monodromy of {H}amiltonian systems with complexity 1 torus actions,
\newblock \emph{J. Geom. Phys.}, \textbf{115} (2017), 104--115.

\bibitem{Eliasson-thesis}
\newblock L.~H. Eliasson,
\newblock \emph{{H}amiltonian systems with Poisson commuting integrals},
\newblock PhD thesis, University of Stockholm, 1984.

\bibitem{gaudin}
\newblock M.~Gaudin,
\newblock Diagonalisation d'une classe d'hamiltoniens de spin,
\newblock \emph{J. Phys. France}, \textbf{37} (1976), 1087--1098,
\newblock DOI: 10.1051/jphys:0197600370100108700.

\bibitem{GuilStern1982}
\newblock V.~Guillemin and S.~Sternberg,
\newblock Convexity properties of the moment mapping,
\newblock \emph{Invent. Math.}, \textbf{67} (1982), 491--513.

\bibitem{HSS}
\newblock S.~Hohloch, S.~Sabatini and D.~Sepe,
\newblock From compact semi-toric systems to {H}amiltonian {$S^1$}-spaces,
\newblock \emph{Discrete Contin. Dyn. Syst.}, \textbf{35} (2015), 247--281.

\bibitem{HSS-vertical}
\newblock S.~Hohloch, S.~Sabatini, D.~Sepe and M.~Symington,
\newblock Vertical almost toric systems,
\newblock ArXiv:1706.09935.

\bibitem{KaPaPeSL2Z}
\newblock D.~M. Kane, J.~Palmer and A.~Pelayo,
\newblock Classifying toric and semitoric fans by lifting equations from
  $\rm{SL}_2(\mathbb{Z})$,
\newblock \emph{SIGMA Symmetry Integrability Geom. Methods Appl.}, \textbf{14}
  (2018), 016, 43 pages.

\bibitem{KaPaPeminimal}
\newblock D.~M. Kane, J.~Palmer and A.~Pelayo,
\newblock Minimal models of compact symplectic semitoric manifolds,
\newblock \emph{J. Geom. Phys.}, \textbf{125} (2018), 49--74.

\bibitem{karshon}
\newblock Y.~Karshon,
\newblock Periodic {H}amiltonian flows on four-dimensional manifolds,
\newblock \emph{Mem. Amer. Math. Soc.}, \textbf{141} (1999), viii+71.

\bibitem{LFPecoupledangular}
\newblock Y.~Le~Floch and A.~Pelayo,
\newblock Symplectic geometry and spectral properties of classical and quantum
  coupled angular momenta,
\newblock ArXiv:1607.05419.

\bibitem{LFPeVN2016}
\newblock Y.~Le~Floch, A.~Pelayo and S.~V\~u Ng\d{o}c,
\newblock Inverse spectral theory for semiclassical {J}aynes-{C}ummings
  systems,
\newblock \emph{Math. Ann.}, \textbf{364} (2016), 1393--1413.

\bibitem{MarsdenRatiu-book}
\newblock J.~Marsden and T.~Ratiu,
\newblock \emph{Introduction to mechanics and symmetry}, vol.~17 of Texts in
  Applied Mathematics,
\newblock 2nd edition,
\newblock Springer-Verlag, New York, 1999,
\newblock \urlprefix\url{https://doi.org/10.1007/978-0-387-21792-5},
\newblock A basic exposition of classical mechanical systems.

\bibitem{miranda-zung}
\newblock E.~Miranda and N.~T. Zung,
\newblock Equivariant normal form for nondegenerate singular orbits of
  integrable {H}amiltonian systems,
\newblock \emph{Ann. Sci. \'Ecole Norm. Sup. (4)}, \textbf{37} (2004),
  819--839.

\bibitem{PalmerJGP2017}
\newblock J.~Palmer,
\newblock Moduli spaces of semitoric systems,
\newblock \emph{J. Geom. Phys.}, \textbf{115} (2017), 191--217.

\bibitem{Pesurvey}
\newblock A.~Pelayo,
\newblock Hamiltonian and symplectic symmetries: an introduction,
\newblock \emph{Bull. Amer. Math. Soc. (N.S.)}, \textbf{54} (2017), 383--436.

\bibitem{PeRaVN2015}
\newblock {\'A}.~Pelayo, T.~Ratiu and S.~{V\~{u} Ng\d{o}c},
\newblock The affine invariant of proper semitoric integrable systems,
\newblock \emph{Nonlinearity}, \textbf{30} (2017), 3993.

\bibitem{PVNinventiones}
\newblock A.~Pelayo and S.~V\~u Ng\d{o}c,
\newblock Semitoric integrable systems on symplectic 4-manifolds,
\newblock \emph{Invent. Math.}, \textbf{177} (2009), 571--597.

\bibitem{PVNacta}
\newblock A.~Pelayo and S.~V\~u Ng\d{o}c,
\newblock Constructing integrable systems of semitoric type,
\newblock \emph{Acta Math.}, \textbf{206} (2011), 93--125.

\bibitem{PVNsurvey}
\newblock A.~Pelayo and S.~V\~u Ng\d{o}c,
\newblock Symplectic theory of completely integrable {H}amiltonian systems,
\newblock \emph{Bull. Amer. Math. Soc. (N.S.)}, \textbf{48} (2011), 409--455,
\newblock \urlprefix\url{https://doi.org/10.1090/S0273-0979-2011-01338-6}.

\bibitem{PeVNcomm-mathphys}
\newblock A.~Pelayo and S.~V\~u Ng\d{o}c,
\newblock Hamiltonian dynamics and spectral theory for spin-oscillators,
\newblock \emph{Comm. Math. Phys.}, \textbf{309} (2012), 123--154.

\bibitem{petrera}
\newblock M.~Petrera,
\newblock \emph{Integrable Extensions and Discretizations of Classical Gaudin
  Models},
\newblock PhD thesis, Dipartimento di Fisica, Universit{\`a} degli Studi di
  Roma Tre, 2007.

\bibitem{RaWaZu2017}
\newblock T.~Ratiu, C.~Wacheux and N.~T. Zung,
\newblock Convexity of singular affine structures and toric-focus integrable
  hamiltonian systems,
\newblock ArXiv:1706.01093.

\bibitem{SaZh-PhysLettersA}
\newblock D.~A. Sadovski{\'\i} and B.~I. Z\^hilinski{\'\i},
\newblock Monodromy, diabolic points, and angular momentum coupling,
\newblock \emph{Phys. Lett. A}, \textbf{256} (1999), 235--244.

\bibitem{Sy2003}
\newblock M.~Symington,
\newblock Four dimensions from two in symplectic topology,
\newblock in \emph{Topology and geometry of manifolds ({A}thens, {GA}, 2001)},
  vol.~71 of Proc. Sympos. Pure Math.,
\newblock Amer. Math. Soc., Providence, RI, 2003,
\newblock 153--208.

\bibitem{VuNgoc03}
\newblock S.~V\~u Ng\d{o}c,
\newblock On semi-global invariants for focus-focus singularities,
\newblock \emph{Topology}, \textbf{42} (2003), 365--380.

\bibitem{San-book}
\newblock S.~V\~u Ng\d{o}c,
\newblock \emph{Syst\`emes int\'egrables semi-classiques: du local au global},
  vol.~22 of Panoramas et Synth\`eses [Panoramas and Syntheses],
\newblock Soci\'et\'e Math\'ematique de France, Paris, 2006.

\bibitem{VuNgoc07}
\newblock S.~V\~u Ng\d{o}c,
\newblock Moment polytopes for symplectic manifolds with monodromy,
\newblock \emph{Adv. Math.}, \textbf{208} (2007), 909--934.

\bibitem{Wacheux-thesis}
\newblock C.~Wacheux,
\newblock \emph{Syst\'{e}mes int\'{e}grables semi-toriques et polytopes
  moment},
\newblock PhD thesis, Universit\'{e} de Rennes 1, 2013.

\bibitem{Williamson1936}
\newblock J.~Williamson,
\newblock On the algebraic problem concerning the normal form of linear
  dynamical systems,
\newblock \emph{Amer. J. Math.}, \textbf{58} (1936), 141--163.

\end{thebibliography}


\end{document}